%% file: support.tex
\documentclass[letterpaper,11pt]{article}

\usepackage[%dvips,
            colorlinks=true,
            citecolor=red,
            linkcolor=blue,
            anchorcolor=red,
            urlcolor=blue
            ]{hyperref}
						\usepackage{hyperref}

\usepackage{ece-paper}

\graphicspath{ {./figures/} }

\newcommand{\Th}{^{\rm th}}

\title{Chebyshev polynomials, moment matching, and optimal estimation of the unseen}
\author{Yihong Wu and Pengkun Yang\thanks{The authors are with
        the Department of Electrical and Computer Engineering and the Coordinated Science Lab, University of Illinois at Urbana-Champaign, Urbana, IL, \texttt{\{yihongwu,pyang14\}@illinois.edu}.}}

\date{\today}

\begin{document}

\maketitle

\begin{abstract}
\input{abstract}
%    whenever $\frac{k}{}\ll n \ll 
    % We also study the closely related species problem where the goal is to estimate the number of distinct colors in an urn containing $k$ balls from repeated draws.   While achieving an additive error proportional to $k$ still requires $ \Omega(\frac{k}{\log k}) $ samples, we show that with $ \Theta(k) $ samples one can strictly outperform a general support size estimator using interpolating polynomials.
\end{abstract}

\input{main}

%\bibliographystyle{alpha}
%\bibliography{../strings,../refs}

\end{document}

%% file: abstract.tex
    We consider the problem of estimating the support size of a discrete distribution whose minimum non-zero mass is at least $ \frac{1}{k}$. Under the independent sampling model, we show that the sample complexity, i.e., the minimal sample size to achieve an additive error of $\epsilon k$ with probability at least 0.1 is within universal constant factors of  $ \frac{k}{\log k}\log^2\frac{1}{\epsilon} $, which improves the state-of-the-art result of $ \frac{k}{\epsilon^2 \log k} $ in \cite{VV13}. Similar characterization of the minimax risk is also obtained. Our procedure is a linear estimator based on the Chebyshev polynomial and its approximation-theoretic properties, which can be evaluated in $O(n+\log^2 k)$ time and attains the sample complexity within a factor of six asymptotically. The superiority of the proposed estimator in terms of accuracy, computational efficiency and scalability is demonstrated in a variety of synthetic and real datasets. 

%% file: main.tex
\section{Introduction}
\label{sec:intro}

\subsection{Model}
\label{sec:model}

Estimating the support size of a distribution from data is a classical problem in statistics with widespread applications.
For example, a major task for ecologists is to estimate the number of species \cite{FCW43} from field experiments; linguists are interested in estimating the vocabulary size of Shakespeare based on his complete works \cite{McNeil73,ET76,TE87}; in population genetics it is of great interest to estimate the number of different alleles in a population \cite{HW01}.
% More practical applications in companies, governments and business can be found in \cite{goodman1949}.
Estimating the support size is equivalent to estimating the number of unseen symbols, which is particularly challenging when the sample size is relatively small compared to the total population size, since a significant portion of the population are never observed in the data.
% apart from the observed species, we need to know how many species are never observed.

%% ------------------------------------------------------ %%
%% Problem formulation                                    %%
%% ------------------------------------------------------ %%
We adopt the following statistical model \cite{BO79,RRSS09}.
%same problem formulation as in \cite{RRSS09}.
Let $ P $ be a discrete distribution over some countable alphabet. Without loss of generality, we assume the alphabet is $ \naturals $ and denote $ P=(p_1,p_2,\dots) $.
Given $ n $ \iid samples $ X\triangleq(X_1,\dots,X_n) $ drawn from $P$, the goal is to estimate the support size
\begin{equation}
    S(P)\triangleq \sum_{i}\indc{p_i>0}.
    \label{eq:supp}
\end{equation}
To estimate the distribution or its functionals, a sufficient statistic is the \emph{histogram} of the samples, denoted by $ N=(N_1,N_2,\dots) $ and 
\begin{equation}
N_i=\sum_{j=1}^{n}\indc{X_j=i}.         
        \label{eq:histogram}
\end{equation}
Therefore $ N$ has a multinomial distribution with parameter $n$ and $P$.
For estimating the support size (or other permutation-invariant functional of the distribution), 
the \emph{fingerprints} form a sufficient statistic which is a further summary of the histogram $N$, which are defined as
\begin{equation}
h_j = \sum_i \indc{N_i = j},    
        \label{eq:fingerprint}
\end{equation}
i.e., the number of items that appear exactly $j$ times.

It is clear that unless we impose further assumptions on the distribution $P$, it is impossible to estimate $S(P)$ within a given accuracy, for otherwise there can be arbitrarily many masses in the support of $P$ that never occur in the samples with high probability and the risk for estimating $S(P)$ is obviously infinite. To prevent the triviality, a conventional assumption \cite{RRSS09} is to impose a lower bound on the non-zero probabilities. Therefore we restrict our attention to the parameter space $ \calD_k $, which consists of all probability distributions on $ \naturals $ whose minimum non-zero mass is at least $ \frac{1}{k} $; consequently $ S(P)\le k $ for any $P \in \calD_k$.
The decision-theoretic fundamental limit of this problem is given by the \emph{minimax risk}:
\begin{equation}
    R^*(k,n)\triangleq \inf_{\hat S}\sup_{P\in\calD_k}\Expect[\ell(\hat S,S)],
    \label{eq:risk}
\end{equation}
where the loss function $ \ell(\hat S,S)\triangleq (\frac{\hat S- S}{k})^2 $ is the normalized mean squared error (MSE) and $\hat S$ is an integer-valued estimator measurable with respect to the samples $X_1,\dots,X_n \iiddistr P$.
% where $ \calD_k $ is the set the all probability distribution on $ \naturals $ with minimum non-zero mass $ \frac{1}{k} $, and 
% Without any lower bound on the non-zero mass, it is impossible to estimate the support size and the minimax risk is trivially $ k^2/4 $ (achieved by $ \hat S=k/2 $). 
% \nb{metaphor for fun: needle in a haystack.}
% The problem defined above is equivalent to the one with arbitrary minimum non-zero mass but oblivious to alphabet size
% since the minimum non-zero mass already yields an upper bound of true support size.
% See \prettyref{app:equiv} for details.

\subsection{Main results}
\label{sec:main}
%% ------------------------------------------------------ %%
%% Main theorem, discussion, comparison                   %%
%% ------------------------------------------------------ %%
Our first main result is the following characterization of the minimax risk:
\begin{theorem}
    \label{thm:main}
    For all $k,n \geq 2$, 
    \begin{equation}
%        \log \frac{1}{R^*(k,n)} \asymp \sqrt{\frac{n\log k}{k}}\vee \frac{n}{k} \vee 1.
R^*(k,n) = \exp\pth{-\Theta\pth{\sqrt{\frac{n\log k}{k}}\vee \frac{n}{k} \vee 1}}.
        \label{eq:main}
    \end{equation}
    Furthermore, if $ \frac{k}{\log k}\ll n\ll k\log k $, as $ k\rightarrow \infty $,
    \begin{equation}
    \label{eq:main-asymp}
        \exp\pth{-(\sqrt{2}e+o(1))\sqrt{\frac{n\log k}{k}}}\le R^*(k,n)\le \exp\pth{-(1.579+o(1))\sqrt{\frac{n\log k}{k}}}
    \end{equation}
\end{theorem}

% The lower and upper bound part of \prettyref{thm:main} is presented in \prettyref{sec:minimax} and \ref{sec:optimal}, respectively.
To interpret the rate of convergence in \prettyref{eq:main}, we consider three cases:
\begin{description}
        \item[Simple regime] $ n \gtrsim k\log k $: we have $ R^*(k,n) = \exp(-\Theta(\frac{n}{k})) $ 
which can be achieved by the simple plug-in estimator
\begin{equation}
    \label{eq:splug}
    \Splug\triangleq \sum_{i} \indc{N_i>0},
\end{equation}
that is, the number of observed symbols. 
Furthermore, if $\frac{n}{k \log k}$ exceeds a sufficiently large constant, all symbols are present in the data and 
$\Splug$ is in fact exact with high probability, namely, $\Prob[\Splug \neq S]\le \Expect(\Splug - S)^2 \to 0$. This can be understood as the classical coupon collector's problem (cf.~\eg, \cite{MU06}).
%However, when $n \ll k \log k$, $\Splug$ grossly underestimate the true support size.

\item[Non-trivial regime] $\frac{k}{\log k} \ll n\ll k\log k $: 
In this case the samples are relatively scarce and the naive plug-in estimator grossly underestimate the true support size as many symbols are simply not observed. Nevertheless, accurate estimation is still possible and the optimal rate of convergence is given by $ R^*(k,n) = \exp(-\Theta(\sqrt{\frac{n\log k}{k}})) $.
%, which vanishes as long as $n \gg \frac{k}{\log k}$.
This can be achieved by a linear estimator based on the Chebyshev polynomial and its approximation-theoretic properties. Although more sophisticated than the plug-in estimator, this procedure can be evaluated in $O(n+\log^2 k)$ time.
% computational complexity is linear in the number of samples. 

\item[Impossible regime] $n \lesssim \frac{k}{\log k}$: no consistent estimator exists.
\end{description}

% and procedures more sophisticated than the plug-in estimator are needed to attain this rate.
% By \prettyref{rmk:rounding} we shall assume without loss of generality that $ \hat S $ belongs to integers.
% Then $\Prob[\hat S \neq S]\le \Expect(\hat S - S)^2$.
% Therefore $\Prob[\hat S \neq S] \to 0$ if $n = \Omega( k \log k)$.
% The transition from non-trivial regime to trivial regime can be interpreted as coupon collector's problem \cite{MR10} where the minimum mass is also $1/k$ by definition.

Next we discuss the \emph{sample complexity} of estimating the support size, which is defined as follows:
\begin{equation}
        n^*(k,\epsilon) \triangleq \min\{n \geq 0\colon \exists \hat S, \text{ s.t. } \Prob[|\hat S - S(P)| \ge \epsilon k] \leq 0.1,  \forall P\in \calD_k \},
        \label{eq:nstar}
\end{equation}
where $\hat S$ is an integer-valued estimator measurable with respect to the samples $X_1,\ldots,X_n\iiddistr P$.
%denoted by $n^*(k,\epsilon)$, which is defined as the minimal sample size $n$ such that there exists an integer-valued estimator $\hat S$  
%measurable to the samples $X_1,\ldots,X_n\iiddistr P$, 
%such that $\Prob[|\hat S - S(P)| \ge \epsilon k] \leq 0.1$ for any $P\in \calD_k$.
Clearly, since $\hat S - S$ is an integer, the only interesting case is $\epsilon \geq \frac{1}{k}$, with $\epsilon=\frac{1}{k}$ corresponding to the exact estimation of the support size since $ |\hat S - S| < 1 $ is equivalent to $ \hat S = S $.
%.\footnote{The case $ \epsilon=\frac{1}{k} $ corresponds to the exact estimation of the support size since $ |\hat S - S| < 1 $ is equivalent to $ \hat S = S $ and to achieve this requires $n = \Omega(k\log k)$.}
Furthermore, since $S(P)$ takes values in $[k]$, $n^*(k,\frac{1}{2})=0$ by definition.
The next result characterizes the sample complexity within universal constant factors that are 
within a factor of six asymptotically.
\begin{theorem}
    \label{thm:sample}
    Fix a constant $c_0 < \frac{1}{2}$.
    For all $\frac{1}{k}\le \epsilon \leq c_0$,
%    \nbr{$\frac{1}{2} - [...]$}
    \begin{equation}
        n^*(k,\epsilon) \asymp \frac{k}{\log k}\log^2\frac{1}{\epsilon}.
        \label{eq:sample-complexity}
    \end{equation}
    Furthermore, if $ \epsilon\rightarrow 0 $ and $\epsilon = k^{o(1)}$,
%    $ \log \frac{1}{\epsilon} =o(\log k) $, 
    as $ k\rightarrow \infty $,
    \begin{equation}
        \frac{1+o(1)}{2e^2}\frac{k}{\log k}\log^2\frac{1}{\epsilon}\le n^*(k,\epsilon)\le \frac{1+o(1)}{2.494}\frac{k}{\log k}\log^2\frac{1}{\epsilon}.
        \label{eq:sample-complexity-asymp}
    \end{equation}
\end{theorem}
Compared to \prettyref{thm:main}, the only difference is that here we are dealing with the zero-one loss $\indc{|S-\hat S| \ge \epsilon k}$ instead of the quadratic loss $(\frac{S-\hat S}{k})^2$.
In the proof we shall obtain upper bound for the quadratic risk and lower bound for the zero-one loss, thereby proving both \prettyref{thm:main} and \ref{thm:sample} simultaneously.
Furthermore, the choice of 0.1 as the probability of error in the definition of the sample complexity is entirely arbitrary; replacing it by $1-\delta$ for any constant $\delta \in (0,1)$ only affect $n^*(k,\epsilon)$ up to constant factors.\footnote{Specifically, upgrading the confidence to $1-\delta$ can be achieved by oversampling by merely a factor of $\log \frac{1}{\delta}$: Let $T = \log \frac{1}{\delta}$. With $nT$ samples, divide them into $T$ batches, apply the $n$-sample estimator to each batch and aggregate by taking the median. Then Hoeffding's inequality implies the desired confidence.}

\subsection{Previous work}
%The problem of support size estimation has been previously investigated by Raskhodnikova et al.~\cite{RRSS09} and Valiant-Valiant \cite{VV10,VV11,VV13}.
%Most of the classical literature mentioned in \prettyref{sec:model} either deals with the classical fixed-$k$-large-$n$ asymptotics, or  focus on specific family of distributions such as uniform or Zipf distribution \cite{DR80,McNeil73}. Other procedures such as Chao-Lee's estimator, despite their practical popularity, lack theoretical guarantees in terms of worst-case estimation error or sample complexity.

There is a vast amount of literature devoted to the support size estimation problem.
In parametric settings, the data generating distribution is assumed to belong to certain parametric family such as uniform or Zipf  \cite{LP56,McNeil73,DR80} and traditional estimators, such as maximum likelihood estimator and minimum variance unbiased estimator, are frequently used
 \cite{harris1968,MS82,Samuel68,ET76,LP56,HW01}  -- see the extensive surveys \cite{BF93,Gandolfi-Sastri04}.
When difficult to postulate or justify a suitable parametric assumption, various nonparametric approaches are adopted such as the Good-Turing estimator \cite{Good1953,Robbins68} and variants due to Chao and Lee \cite{Chao84,Chao92}, Jackknife estimator \cite{BO79}, empirical Bayes approach (\eg, Good-Toulmin estimator \cite{GT56}), one-sided estimator \cite{ML07}.
Despite their practical popularity, little is known about the performance guarantee of these estimators, let alone their optimality. 
Next we discuss provable results assuming the independent sampling model in \prettyref{sec:model}.

 For the naive plug-in estimator \prettyref{eq:splug}, it is easy to show (see \prettyref{prop:main-ub-plug}) that to estimate $S(P)$ within $\pm \epsilon k$ the minimal required number of samples is $\Theta(k \log \frac{1}{\epsilon})$, which scales logarithmically in $\frac{1}{\epsilon}$ but linearly in $k$, the same scaling for estimating the distribution $P$ itself.
 Recently Valiant and Valiant \cite{VV11} showed that the sample complexity is in fact sub-linear in $k$; however, the performance guarantee of the proposed estimators are still far from being optimal.
Specifically, an estimator based on a linear program that is a modification of \cite[Program 2]{ET76} is proposed and shown to achieve $n^*(k,\epsilon) \lesssim \frac{k}{\epsilon^{2+\delta}\log k} $ for any arbitrary $ \delta>0 $ \cite[Corollary 11]{VV11}, which has subsequently been improved to $ \frac{k}{\epsilon^2\log k} $ in \cite[Theorem 2, Fact 9]{VV13}. 
The lower bound $ n^*(k,\epsilon) \gtrsim \frac{k}{\log k} $ in \cite[Corollary 9]{VV10} is optimal in $k$ but provides no dependence on $ \epsilon $.
These results show that the optimal scaling in terms of $k$ is $\frac{k}{\log k}$ but the dependence on the accuracy $\epsilon$ is $\frac{1}{\epsilon^2}$, which is even worse than the plug-in estimator. 
From \prettyref{thm:sample} we see that the dependence on $\epsilon$ can be improved from polynomial to polylogarithmic $\log^2\frac{1}{\epsilon}$, which turns out to be optimal.
Furthermore, this can be attained by a linear estimator which is far more scalable than linear programming on massive datasets (see the experiment on New York Times datasets of one billion words in \prettyref{sec:exp}).
%Moreover, in the regime that $ \log(1/\epsilon)=o(\log k) $ our estimator needs $ \frac{0.401k}{\log k}\log^2\frac{1}{\epsilon} $ samples and achieves the optimal constant of $ n^*(k,\epsilon) $ within a factor of six.

%It is easy to show that the sample complexity for the naive plug-in estimator \prettyref{eq:splug} is $\Theta(k \log \frac{1}{\epsilon})$, which scales linearly in $k$ and logarithmically in $\frac{1}{\epsilon}$.
%Valiant and Valiant proposed an esitmator based on linear programming that is similar to \cite{ET76} and showed it achieves $n^*(k,\epsilon) \lesssim \frac{k}{\epsilon^{2+\delta}\log k} $ samples for any arbitrary $ \delta>0 $ \cite[Corollary 11]{VV11}, which has subsequently been improved to $ \frac{k}{\epsilon^2\log k} $ in \cite[Theorem 2, Fact 9]{VV13}.
%The lower bound $ n^*(k,\epsilon) \gtrsim \frac{k}{\log k} $ in \cite[Corollary 9]{VV10} is optimal in $k$ but provides no dependence on $ \epsilon $.
%In fact, it is easy to show that the plug-in estimator \prettyref{eq:splug} achieves $\Theta(k \log \frac{1}{\epsilon})$, which scales linearly in $k$ but logarithmically in $\frac{1}{\epsilon}$.
%From \prettyref{thm:sample} we see that the dependence of sample complexity on the accuracy $\epsilon$ can be improved from polynomial to polylogarithmic, which turns out to be optimal.
%Moreover, in the regime that $ \log(1/\epsilon)=o(\log k) $ our estimator needs $ \frac{0.401k}{\log k}\log^2\frac{1}{\epsilon} $ samples and achieves the optimal constant of $ n^*(k,\epsilon) $ within a factor of six.

A closely related problem is the \emph{distinct elements} problem, where the goal is to estimate the number of distinct colors based on repeated draws from in an urn consisting of $k$ colored balls. For sampling with replacement, this can be viewed as a restricted case of the model in the present paper, where the distribution $P=(p_i)$ has the special form of $p_i = \frac{k_i}{k}$, with $k_i \in \integers_+$ corresponding to the number of balls of the $i\Th$ color and $\sum_i k_i=k$. 
The sample complexity under multiplicative error, that is, estimating $S(P)$ within a factor of $\alpha$ has been shown to be $\Theta(\frac{k}{\alpha^2})$ in \cite{CCMN00}. For additive error, that is, estimating $S(P)$ within $\pm \epsilon k$, 
a lower bound has been established in  \cite{RRSS09}, which, for constant $\epsilon$, scales as $k^{1 - O(\sqrt{\frac{\log\log k}{\log k}})}$.
This, in turn, implies a lower bound for $n^*(k,\epsilon)$, which is slightly suboptimal compared to the tight bound $\frac{k}{\log k} = k^{1 - \frac{\log\log k}{\log k}}$.

\subsection{Organization}
%       \label{sec:}
The paper is organized as follows: In \prettyref{sec:minimax} we outline the proof for the lower bound part of \prettyref{thm:main} and \ref{thm:sample} and the construction of the least favorable priors.
In \prettyref{sec:optimal} we construct an estimator based on Chebyshev polynomials which achieves the minimax risk and the sample complexity within constant factors.
% the optimal minimax rate
% \nbr{\prettyref{thm:main} does not qualify as minimax rate. So we should say which achieve the sample complexity within a constant factor}.
% \prettyref{sec:species} considers the closely related species problem which can be viewed as a special case of the general support size problem.
In \prettyref{sec:exp} we apply our estimators to both synthetic and real data and compare the performance with existing methodologies.
Proofs of the lower and upper bounds are given in \prettyref{sec:pf-lb} and \ref{sec:pf-ub}, respectively.

%\paragraph{Notations}
\subsection{Notations}
For $ k\in \naturals $, let $ [k]\triangleq \sth{1,\dots,k} $.
The $n$-fold product of a distribution $P$ is denoted by $P^{\otimes n}$.
Let $\Poi(\lambda)$ denote the Poisson distribution with mean $\lambda$ whose probability mass function is denoted by $\poi(\lambda,j) \triangleq \frac{\lambda^j e^{-\lambda}}{j!}, j \geq 0$.
Given a positive random variable $U$, denote the Poisson mixture with respect to the distribution of $U$ by $\expect{\Poi\pth{U}}$, whose probability mass function is given by $\frac{1}{j!} \Expect[ U^j e^{-U}], j \geq 0$.
Let $\Bern(p)=p\delta_1+(1-p)\delta_0$ denote the Bernoulli1i distribution.
The total variation and the Kullback-Leibler divergence between probability measures $P$ and $Q$ are denoted by $\TV(P,Q) \triangleq \frac{1}{2} \int |\diff P-\diff Q|$ and $D(P\|Q) \triangleq  \int \diff P \log \frac{\diff P}{\diff Q}$ respectively.
We use standard big-$O$ notations,
% as well as their counterparts in probability, 
%e.g., for any positive sequences $\{a_n\}$ and $\{b_n\}$, we write $a_n\gtrsim b_n$ or $b_n\lesssim a_n$ when $a_n\geq cb_n$ for some absolute constant $c$.    Finally, we write $a_n \asymp b_n$ when both $a_n\gtrsim b_n$ and $a_n\lesssim b_n$ hold.
e.g., for any positive sequences $\{a_n\}$ and $\{b_n\}$, $a_n=O(b_n)$ or $a_n \lesssim b_n$ if $a_n \leq C b_n$ for some absolute constant $C>0$, 
$a_n=o(b_n)$ or $a_n \ll b_n$ or if $\lim a_n/b_n = 0$.
In order to extract non-asymptotic statements from asymptotic ones, we pay extra attention to $o(1)$ terms. Specifically, we write $o_{\delta}(1)$  as $\delta\to0$ to indicate convergence to zero that is uniform in all other parameters.

\section{Minimax lower bound}
\label{sec:minimax}
% The strategy to prove the lower bound in \prettyref{thm:main} is similar to \cite{WY14}, which is divided into the case $ n\gtrsim k\log k $ and $ n\lesssim k\log k $ in 
% \prettyref{prop:main-lb1} and \ref{prop:main-lb2}, respectively. 
% \begin{prop}
%     For all $ \frac{1}{k}\le \epsilon \le \frac{1}{4} $, and for all $ k,n $
%     \begin{equation}
%         n^*(k,\epsilon)
%         \gtrsim k\log\pth{1+\frac{1-2\epsilon}{4k\epsilon^2}},
%         \qquad
%         R^*(k,n)\gtrsim 1\wedge ke^{-2n/k}\ge 1 \wedge k^2e^{-4n/k}.
%         \label{eq:main-lb1}
%     \end{equation}
%     \label{prop:main-lb1}
% \end{prop}
% \nbr{Replace the sample complexity with previous one?}

% The lower bounds in \prettyref{prop:main-lb1} follow from a simple versus composite hypothesis testing, where one hypothesis is $ \Uniform[k] $ and the other one consists of $ k $ perturbed uniform distributions.
% Their support sizes differ by one while the generated $n$ samples are statistically indistinguishable.
% The best lower bound of $ n^*(k,0) $ via simple versus simple hypothesis testing is $ n^*(k,0)\gtrsim k $, where the two distributions are $ \Bern(0)$ and $ \Bern(\frac{1}{k})$.

% simple application of Le Cam's two-point method \cite{Lecam86} by considering two possible distributions, namely $ \Bern(0)$ and $ \Bern(\frac{1}{k})$.
% Any estimator then necessarily suffers a risk proportional to the statistical distance.
% We remark that the same sample complexity and minimax MSE lower bound also hold in the species problem in \prettyref{sec:species}.

The lower bound argument follows the idea in \cite{LNS99,CL11,WY14} and relies on the generalized Le Cam's lemma involving two composite hypothesis.
In the following we illustrate the main idea for constructing a pair of priors that are near least favorable.

Let $\lambda > 1$. Given unit-mean random variables $ U$ and $U'$ that take values in $\{0\} \cup [1,\lambda]$,
% satisfying $ \expect{U}=\expect{U'}=1 $ 
define the following random vectors
\begin{equation}
    \sfP=\frac{1}{k}(U_1,\dots,U_k),\quad \sfP'=\frac{1}{k}(U_1',\dots,U_k'),
    \label{eq:iid-PP}
\end{equation}
where $ U_i $ and $U_i' $ are \iid~copies of $ U$ and $U' $, respectively.
Although $ \sfP $ and $ \sfP' $ need not be probability distributions, as long as the standard deviation of $U$ and $U'$ are not too big, 
the law of large numbers ensures that with high probability $ \sfP$ and $\sfP' $ lie in a small neighborhood near the probability simplex, which we refer as the set of \emph{approximate} probability distributions. Furthermore, the minimum non-zeros in $\sfP$ and $\sfP'$ are at least $\frac{1}{k}$.
It can be shown that the minimax risk over approximate probability distributions is close to that over the original parameter space $\calD_k$ of probability distributions.
This allows us to use $ \sfP$ and $\sfP' $ as priors and apply Le Cam's method.
Note that both $ S(\sfP) $ and $ S(\sfP') $ are binomially distributed, which, with high probability, differ by the difference in their mean values:
\begin{align*}
  \Expect[S(\sfP)]-\Expect[S(\sfP')]
  =k(\Prob[U>0]-\Prob[U'>0])
  =k(\Prob[U'=0]-\Prob[U=0]).
\end{align*}
If we can establish the impossibility of testing whether data are generated from $\sfP$ or $\sfP'$, the resulting lower bound is proportional to $k(\Prob[U'=0]-\Prob[U=0])$.
% Similar to \cite[Section 2]{WY14}

To simplify the argument we apply the Poissonization technique where the sample size is a $ \Poi(n) $ random variable instead of a fixed number $n$.
This provably does not change the statistical nature of the problem due to the concentration of $\Poi(n)$ around its mean $ n $.
Under Poisson sampling, the histograms \prettyref{eq:histogram} still constitute a sufficient statistic, which are distributed as $ N_i\inddistr \Poi(np_i) $, as opposed to multinomial distribution in the fixed-sample-size model.
Therefore through the \iid construction in \prettyref{eq:iid-PP}, $ N_i\iiddistr \Expect[\Poi(\frac{n}{k}U)] $ or $ \Expect[\Poi(\frac{n}{k}U')] $.
Then Le Cam's lemma is applicable if $   \TV(\Expect[\Poi(\frac{n}{k} U)]^{\otimes k},\Expect[\Poi(\frac{n}{k} U')]^{\otimes k})$ is strictly bounded away from one, 
%the indistinguishability follows if the total variation distance between the two product Poisson mixtures is strictly bounded away from one, 
for which it suffices to show
\begin{equation}
    \TV(\Expect[\Poi(nU/k)],\Expect[\Poi(nU'/k)])\le \frac{c}{k},
    \label{eq:tv-bd}
\end{equation}
for some constant $ c<1 $.

The above construction provides a recipe for the lower bound.
To optimize the ingredients it boils down to the following optimization problem (over one-dimensional probability distributions):
Construct two priors $ U,U' $ with unit mean that maximize the difference $ \prob{U'=0}-\prob{U=0} $ subject to the total variation distance constraint \prettyref{eq:tv-bd}, which, in turn, can be guaranteed by \emph{moment matching}, \ie, ensuring $ U$ and $U' $ have identical first $L$ moments for some large $L$, and the $L_\infty$-norms $ U,U' $ are not too large. % \nbr{the reader wonders why $L+1$ not $L$?}
% up to some large degree $ L+1 $.
% Together with the $ L_\infty $-norm of $ U,U' $ we obtain the desired total variation distance.
To summarize, our lower bound entails solving the following optimization problem:
\begin{equation}
    \begin{aligned}
        % \calF_L(\lambda) \triangleq 
        \sup & ~ \Prob[U'=0]-\Prob[U=0]  \\
        \text{s.t.}     
        & ~ \Expect[U] = \Expect[U']=1 \\
        & ~ \Expect[U^j] = \Expect[U'^j], \quad j = 1,\ldots,L \\
        & ~ U,U' \in \sth{0}\cup[1, \lambda].
    \end{aligned}
    \label{EQ:FL}
\end{equation}
The final lower bound is obtained from \prettyref{EQ:FL} by choosing $L \asymp \log k$ and $\lambda \asymp \frac{k\log k}{n}$. 

In order to evaluate the infinite-dimensional linear programming problem \prettyref{EQ:FL}, 
% similar to \cite{WY14} dealing the entropy function $ x\mapsto x\log \frac{1}{x} $, 
by considering its dual program we show (in \prettyref{app:opt-UX}) that \prettyref{EQ:FL} coincides exactly with the best uniform approximation error of the function $ x\mapsto\frac{1}{x} $ over the interval $ [1,\lambda] $ by degree-$ (L-1) $ polynomials:
\begin{equation*}
    \inf_{p\in \calP_{L-1}}\sup_{x\in [1,\lambda]}\abs{\frac{1}{x}-p(x)},
\end{equation*}
where $ \calP_{L-1} $ denotes the set of polynomials of degree $ L-1 $.
The problem of best polynomial approximation has been well-studied, cf.~\cite{timan63,DS08};
in particular, the exact formula for the best polynomial that approximates $ x\mapsto\frac{1}{x} $ and the optimal approximation error have been obtained in \cite[Sec. 2.11.1]{timan63}.

Applying the procedure described above, we obtain the following sample complexity lower bound:
\begin{prop}
    \label{prop:main-lb2}
    % If $ 1-2\epsilon\gg \frac{\sqrt{\log k}}{k^{1/4}} $ and $ \epsilon=k^{-o(1)} $,
    Let $ \delta\triangleq \frac{\log\frac{1}{\epsilon}}{\log k}$ and $ \tau\triangleq \frac{\sqrt{\log k}/k^{1/4}}{1-2\epsilon} $. As $ k\rightarrow\infty $, $\delta \to 0$ and $\tau\to 0$, 
    \begin{equation}
        n^*(k,\epsilon) \ge (1-o_{\delta}(1)-o_{k}(1)-o_\tau(1))\frac{k}{2e^2\log k}\log^2\frac{1}{2\epsilon}.
        \label{eq:main-lb2}
    \end{equation}
    Consequently,    if $ \frac{1}{k^c}\le \epsilon \le \frac{1}{2}-c'\frac{\sqrt{\log k}}{k^{1/4}} $ for some constants $ c,c' $ then $ n^*(k,\epsilon)\gtrsim \frac{k}{\log k}\log^2\frac{1}{2\epsilon} $.
\end{prop}
The lower bounds announced in Theorems \ref{thm:main} and \ref{thm:sample} follow from \prettyref{prop:main-lb2} combined with a simple two-point argument. See \prettyref{sec:pf-thm-lb}.

\section{Optimal estimator via Chebyshev polynomials}
\label{sec:optimal}
In this section we prove the upper bound part of \prettyref{thm:main} and describe the rate-optimal support size estimator.
Following the same idea as in the lower bound part, we shall apply the Poissonization technique to simplify the analysis where the sample size is $ \Poi(n) $ instead of a fixed number $ n $ and hence the sufficient statistics $ N=(N_1,\dots,N_k)\inddistr\Poi(np_i) $.
Analogous to \prettyref{eq:risk}, the minimax risk under the Poisson sampling is defined by
\begin{equation}
    \tilde{R}^*(k,n)
    \triangleq \inf_{\hat{S}}\sup_{P\in\calD_k}\Expect[\ell(\hat S,S)].
    \label{eq:risk-poisson}
\end{equation}
Due to the concentration of $ \Poi(n) $ near its mean $ n $, the minimax risk with fixed sample size is close to that under the Poisson sampling, as shown in the following lemma, which allows us to focus on the model using Poissonized sample size.
\begin{lemma}
    For any $ \beta<1 $,
    \begin{equation*}
        R^*(k,n)\le \frac{\tilde{R}^*(k,(1-\beta)n)}{1-\exp(-n\beta^2/2)}.
    \end{equation*}
    \label{lmm:ub-poisson}
\end{lemma}
% Analogous to the lower bound part, we obtain the upper bound in \prettyref{thm:main} by \prettyref{prop:main-ub-plug} for the regime of $ n\gtrsim k\log k $ and \prettyref{prop:main-ub-poly} for $ n\lesssim k\log k $ separately.
In the next proposition, we first analyze the risk of the plug-in estimator $ \Splug $, which yields the optimal upper bound of \prettyref{thm:main} in the regime of $ n\gtrsim k\log k $. This is consistent with the coupon collection intuition explained in \prettyref{sec:main}.
%which performs well when $ n\gtrsim k\log k $ as suggested by the coupon collector's problem.
%The performance of $ \Splug $ yields the optimal upper bound of \prettyref{thm:main} in the regime of $ n\gtrsim k\log k $:
\begin{prop}
    \label{prop:main-ub-plug}
    For all $ n, k\ge 1 $,
    \begin{equation}
        \sup_{P\in\calD_k}\Expect(S(P)-\Splug(N))^2\le k^2e^{-2n/k}+ke^{-n/k},
        \label{eq:main-ub-plug}
    \end{equation}
    where $ N=(N_1,N_2,\dots) $ and $ N_i\inddistr\Poi(np_i) $.
    
    Conversely, for $P$ that is uniform over $[k]$, for any fixed $\delta \in (0,1)$, if $n \leq (1-\delta) k \log \frac{1}{\epsilon}$, then as $k\diverge$,
    \begin{equation}
        \Prob[|S(P)-\Splug(N)|\leq \epsilon k]\leq e^{-\Omega(k^\delta)}.
        \label{eq:plug-converse}
\end{equation}    
%    the sample complexity of the plug-in estimator is at least $ n_{\mathrm{seen}}^*(k,\epsilon)\gtrsim k\log\frac{1}{\epsilon} $ for any $ \epsilon\ge \frac{1}{k} $.
\end{prop}

To specify the optimal estimator in the regime of $ n\lesssim k\log k $, we first introduce Chebyshev polynomials.
Recall that the usual Chebyshev polynomial of degree $L$ is 
\begin{equation}
    T_L(x) = \cos (L\arccos x) = (z^L+z^{-L})/2,
    \label{eq:cheby}
\end{equation}
where $ z $ is the solution of the quadratic equation $ z+z^{-1}=2x $.
Note that $T_L$ is bounded in magnitude by one over the interval $[-1,1]$. 
The shifted and scaled Chebyshev polynomial over the interval $ [l,r] $ is given by
\begin{equation}
    P_L(x)
    =-\frac{T_L(\frac{2x-r-l}{r-l})}{T_L(\frac{-r-l}{r-l})}
    % =-\frac{\cos (L\arccos (\frac{2}{r-l}x-\frac{r+l}{r-l}))}{\cos (L\arccos (-\frac{r+l}{r-l}))}
    \triangleq \sum_{m=0}^{L}a_mx^m,
    \label{eq:PL}
\end{equation}
which satisfies $ P_L(0)=-1 $ and hence $ a_0=-1 $; the remaining coefficients $a_1,\dots,a_L $ can be obtained from those of the Chebyshev polynomial \cite[2.9.12]{timan63} and the binomial expansion, or more directly,
\begin{equation}
    a_j=\frac{P_L^{(j)}(0)}{j!}=-\pth{\frac{2}{r-j}}^j\frac{1}{j!}\frac{T_L^{(j)}(-\frac{r+l}{r-l})}{T_L(-\frac{r+l}{r-l})}.
    \label{eq:aj}
\end{equation}

Let
\begin{equation}
    g_L(j)=
    \begin{cases}
        \frac{a_jj!}{n^j}+1,&  j\le L,\\
        1,&j> L.
    \end{cases}
    \label{eq:gL}
\end{equation}
Obviously $ g_L(0)=0 $ since $ a_0=-1 $ by design.
We Define our estimator by
\begin{equation}
    \hat S=\sum_{i}g_L(N_i).
    \label{eq:hatS}
\end{equation}
We proceed to explain the reasoning behind the estimator \prettyref{eq:hatS}.
Note that the bias is $\Expect[\hat S - S] = \sum_i \expect{g_L(N_i)-\indc{p_i>0}}$.
Since $ g_L(0)=0 $ and $ g_L(j)=1 $ for $ j>L $, each term in the bias can be written as
\begin{align}
  \expect{g_L(N_i)-\indc{p_i>0}}
  =&\expect{(g_L(N_i)-1)\indc{p_i>0}\indc{N_i\le L}}\nonumber\\
  =&\sum_{j=0}^{L}e^{-np_i}\frac{(np_i)^j}{j!}\frac{a_jj!}{n^j}\indc{p_i>0}
     =e^{-np_i}P_L(p_i)\indc{p_i>0}
     \label{eq:bias-term}
\end{align}
where $P_L$ is the degree-$L$ polynomial defined in \prettyref{eq:PL}.

Let
\begin{equation}
    L \triangleq \floor{c_0\log k}, \quad r\triangleq\frac{c_1\log k}{n}, \quad l\triangleq\frac{1}{k},
    \label{eq:constants-ref}
\end{equation} 
where $ c_0<c_1 $ are constants to be specified.
The main intuition is that since $c_0<c_1$, 
then with high probability, whenever $ N_i\le L=\floor{c_0\log k} $ the corresponding mass must satisfy $ p_i\le \frac{c_1\log k}{n} $.
That is, if $ p_i>0 $ and $ N_i\le L $ then $ p_i\in [\frac{1}{k},\frac{c_1\log k}{n}] $, and hence $ P_L(p_i) $ is bounded by the sup-norm of $ P_L $ over the interval $ [\frac{1}{k},\frac{c_1\log k}{n}] $.
In view of the property of Chebyshev polynomials \cite[Ex.~2.13.14]{timan63}, \prettyref{eq:PL} is the unique degree-$L$ polynomial that passes through the point $(0,-1)$ and deviates the least from zero over $ [\frac{1}{k},\frac{c_1\log k}{n}] $.
This explains the coefficients \prettyref{eq:gL} which are chosen to minimize the bias.
%We remark that when $ n\gtrsim k\log k $ the approximation interval $ [l,r] $ vanishes, in which case we define \nbr{define or does the previous definition reduces to this?} $ P_L(x)=-1 $ and our estimator \prettyref{eq:hatS} coincides with the plug-in estimator.

% \nbr{Remove? Next proposition also holds for $ \tilde{S} $.}
% Finally, in view of the fact that $ S(P)\in[0,k] $, we define our estimator by 
% \begin{equation}
%     \hat S\triangleq (\tilde S \vee 0)\wedge k .
%     \label{eq:Shat}
% \end{equation}
The next proposition gives an upper bound of the quadratic risk of our estimator \prettyref{eq:hatS}:
\begin{prop}
Let $ c_0 = 0.558$ and $c_1=0.5 $.
    As $ \delta\triangleq \frac{n}{k\log k}\rightarrow 0 $ and $ k\diverge $,
    \begin{equation}
        \sup_{P\in\calD_k}\Expect(\hat S(N)-S(P))^2\le  4k^2(1+o_{k}(1))\exp\pth{-(2+o_\delta(1))\sqrt{\kappa\frac{n\log k}{k}}},
        \label{eq:main-ub-poly}
    \end{equation}
    where $ N=(N_1,N_2,\dots)\inddistr\Poi(np_i) $, and $\kappa = 2.494$.
    % In particular, if $ n\lesssim k\log k $, then the upper bound is $ 4k^2\exp\pth{-c\sqrt{\frac{n\log k}{k}}} $ for some universal constant $ c $.
    \label{prop:main-ub-poly}
\end{prop}

The minimax upper bounds in Theorems \ref{thm:main} and \ref{thm:sample} follow from combining Propositions \ref{prop:main-ub-plug} and \ref{prop:main-ub-poly}. See \prettyref{sec:pf-thm-ub}.

The estimator \prettyref{eq:hatS} belong to the family of \emph{linear estimators}:
\begin{equation}
    \hat S = \sum_i f(N_i) =\sum_{j \geq 1} f(j) h_j ,       
    \label{eq:linear}
\end{equation} 
which is a linear combination of fingerprints $h_j$'s defined in \prettyref{eq:fingerprint}.
Other notable examples of linear estimators include:
\begin{itemize}
    \item Plug-in estimator \prettyref{eq:splug}: $\Splug = h_1+h_2+\dots$. % , which is also linear with all-one coefficients.
    \item Good-Toulmin estimator \cite{GT56}: for some $t>0$,
    \begin{equation}
        \hat S_{\rm GT} = \Splug + 
        % = \sum_{j \geq 1} (t^j )^{}h_j
        t h_1 - t^2 h_2 + t^3 h_3 - t^4 h_4 + \ldots
        \label{eq:GT}
    \end{equation}
    % \begin{figure}[ht]
    %     \centering
    %     \includegraphics[width=.7\textwidth]{good}
    % \end{figure}
    \item Efron-Thisted estimator \cite{ET76}: for some $t>0$ and $J\in\naturals$,
    \begin{equation}
        \hat S_{\rm ET} = \Splug + \sum_{j=1}^J (-1)^{j+1} t^j b_j h_j,
        \label{eq:ET}
    \end{equation}
    % t b_1 h_1 - t^2 b_2 h_2 + t^3 b_3 h_3 - t^4 b_4 h_4 + \ldots
    where $b_j = \Prob[\Binom(J,1/(t+1)) \geq j]$.
\end{itemize}
By definition, our estimator \prettyref{eq:hatS} can be written as
\begin{equation}
    \hat S =  \sum_{j=1}^L g_L(j) h_j + \sum_{j>L} h_j.
    \label{eq:Shat1}
\end{equation}
% , where the coefficients are $f_j = g_L(j)$ if $j \leq L$ and $1$ otherwise.
% is further histogram of $ N $ with is known as histogram of histogram or fingerprint \cite{VV10}.
% In contrast, the plug-in estimator is also linear with all-one coefficients.
By \prettyref{eq:gL}, $ g_L $ is also a polynomial of degree $ L $, which is oscillating and results in coefficients with alternating signs (see \prettyref{fig:coeffs}).
% it is surprising to see $ g_L $ is oscillating and has alternating signs,
\begin{figure}[hbt]
    \centering
    \includegraphics[width=.6\linewidth]{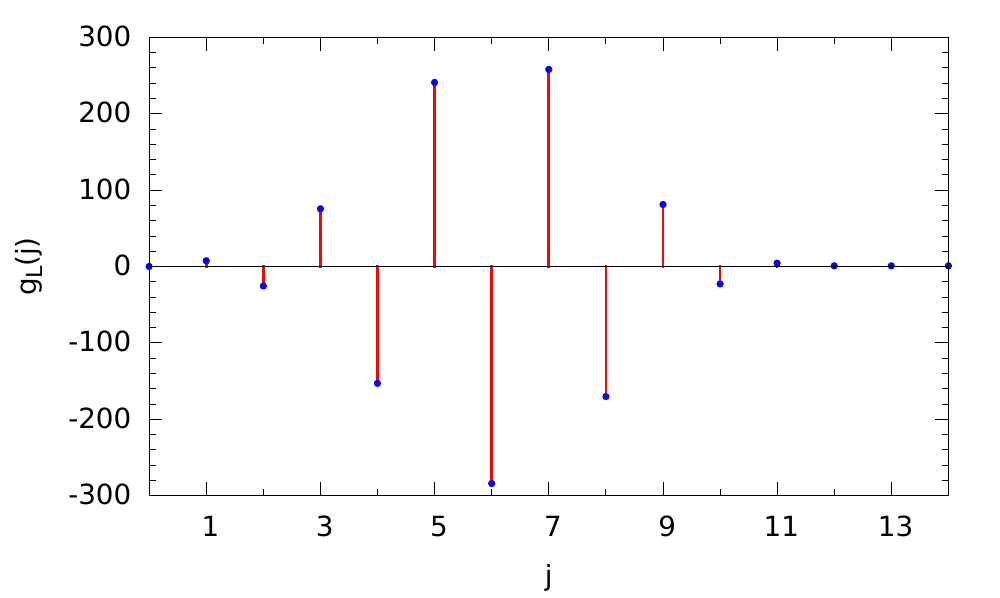}
    \caption{Coefficients of estimator $ g_L $ in \prettyref{eq:gL} with $ c_0=0.45,c_1=0.5 $, $ k=10^6$ and $ n=2\times 10^5$.\label{fig:coeffs} }
\end{figure}
Interestingly, this behavior, although counterintuitive, coincide with many classical estimators, such as \prettyref{eq:GT} and \prettyref{eq:ET}.
% the estimator by Efron-Thisted \cite{ET76} and Jackknife estimator \cite{BO79}, which are also linear estimators with oscillating coefficients.

\begin{remark}[Time complexity]
    \label{rmk:complexity}
    The evaluation of the estimator \prettyref{eq:linear} consists of three parts:
    \begin{enumerate}
        \item Construction of the estimator: $ O(L^2)=O(\log^2k) $, which amounts to computing the coefficients $ f_L(j) $ per \prettyref{eq:aj};
        \item Computing the histograms $ N_i $ and fingerprints $ h_j $: $ O(n) $;
        \item Evaluating the linear combination: $ O(n\wedge k) $, since the number of non-zero terms in the second summation of \prettyref{eq:linear} is at most $ n\wedge k $. 
    \end{enumerate}
    Therefore the total time complexity is $ O(n+\log^2k) $.
\end{remark}

\begin{remark}
    The technique of polynomial approximation has been previously used for estimating non-smooth functions ($L_q$-norms) in Gaussian models \cite{INK87,LNS99,CL11} and more recently for estimating information quantities (entropy and power sums) on large discrete alphabets \cite{WY14,JVHW15}.
    The design principle is to  approximate the non-smooth function on a given interval using algebraic or trigonometric polynomials for which unbiased estimators exist and choose the degree to balance the bias (approximation error) and the variance (stochastic error).
    Note that in general uniform approximation by polynomials is only possible on a compact interval. 
    Therefore, in many situations, the construction of the estimator is a two-stage procedure involving \emph{sample splitting}:
    First, use half of the sample to test whether the corresponding parameter lies in the given interval; 
    Second, use the remaining samples to construct an unbiased estimator for the approximating polynomial if the parameter belongs to the interval or apply plug-in estimators otherwise (see, \eg, \cite{WY14,JVHW15} and \cite[Section 5]{CL11}).

    While the benefit of sample splitting is to make the analysis tractable by capitalizing on the independence of the two subsamples, the downside is obviously sacrificing the statistical accuracy since half of the samples are wasted.
    In the present paper, to estimate the support size, we forgo the sample splitting approach and directly design a linear estimator.
    Instead of using a polynomial as a proxy for the original function and then constructing its unbiased estimator, the best polynomial approximation arises as a natural step in controlling the bias (see \prettyref{eq:bias-term}).
    % \label{rmk:}
\end{remark}

\section{Experiments}
\label{sec:exp}
We evaluate the performance of our estimator on both synthetic and real datasets in comparison with popular existing procedures. 
In the experiments we choose the constants $ c_0=0.45, c_1=0.5 $ in \prettyref{eq:constants-ref}, 
instead of $c_0=0.558$ which is optimized to yield the best rate of convergence in \prettyref{prop:main-ub-poly} under the iid sample model.
 % from the optimal choice in theory.
 The reason for such a choice is that in the real-data experiments the samples are not necessarily generated independently and dependency leads to a higher variance. By choosing a smaller $c_0$, the Chebyshev polynomials have a slightly smaller degree, which results in smaller variance and more robustness to model mismatch.
% It provably requires $ \frac{0.62k}{\log k}\log^2\frac{1}{\epsilon} $ samples. 
%The minimum non-zero mass varies in different experimental settings.
Each experiment is averaged over $ 50 $ independent trials and the standard deviations are shown as error bars.

\paragraph{Synthetic data}
%For synthetic-data experiments the minimum non-zero mass is fixed to be $ 10^{-6} $ and Thus a degree-6 Chebyshev polynomial is applied. Therefore, according to \prettyref{eq:Shat1}, we apply the polynomial estimator $ g_L $ to symbols appearing at most 6 times and the plug-in estimator otherwise. 
We consider data independently sampled from the following distributions, 
(a) the uniform distribution with $ p_i=\frac{1}{k} $, 
(b) Zipf distributions with $ p_i\propto i^{-\alpha} $ and $ \alpha$ being either $1$ or $0.5 $, 
(c) an even mixture of geometric distribution and Zipf distribution where for the first half of the alphabet $ p_i \propto 1/i$ and  for the second half $ p_{i+k/2} \propto (1-\frac{2}{k})^{i-1} $, $ 1\le i\le \frac{k}{2} $.
The alphabet size $ k $ varies in each distribution so that the minimum non-zero mass is roughly $ 10^{-6} $. Accordingly, a degree-6 Chebyshev polynomial is applied. Therefore, according to \prettyref{eq:Shat1}, we apply the polynomial estimator $ g_L $ to symbols appearing at most six times and the plug-in estimator otherwise. 
\begin{figure}[ht]
    \centering
    \includegraphics[width=1.0\linewidth]{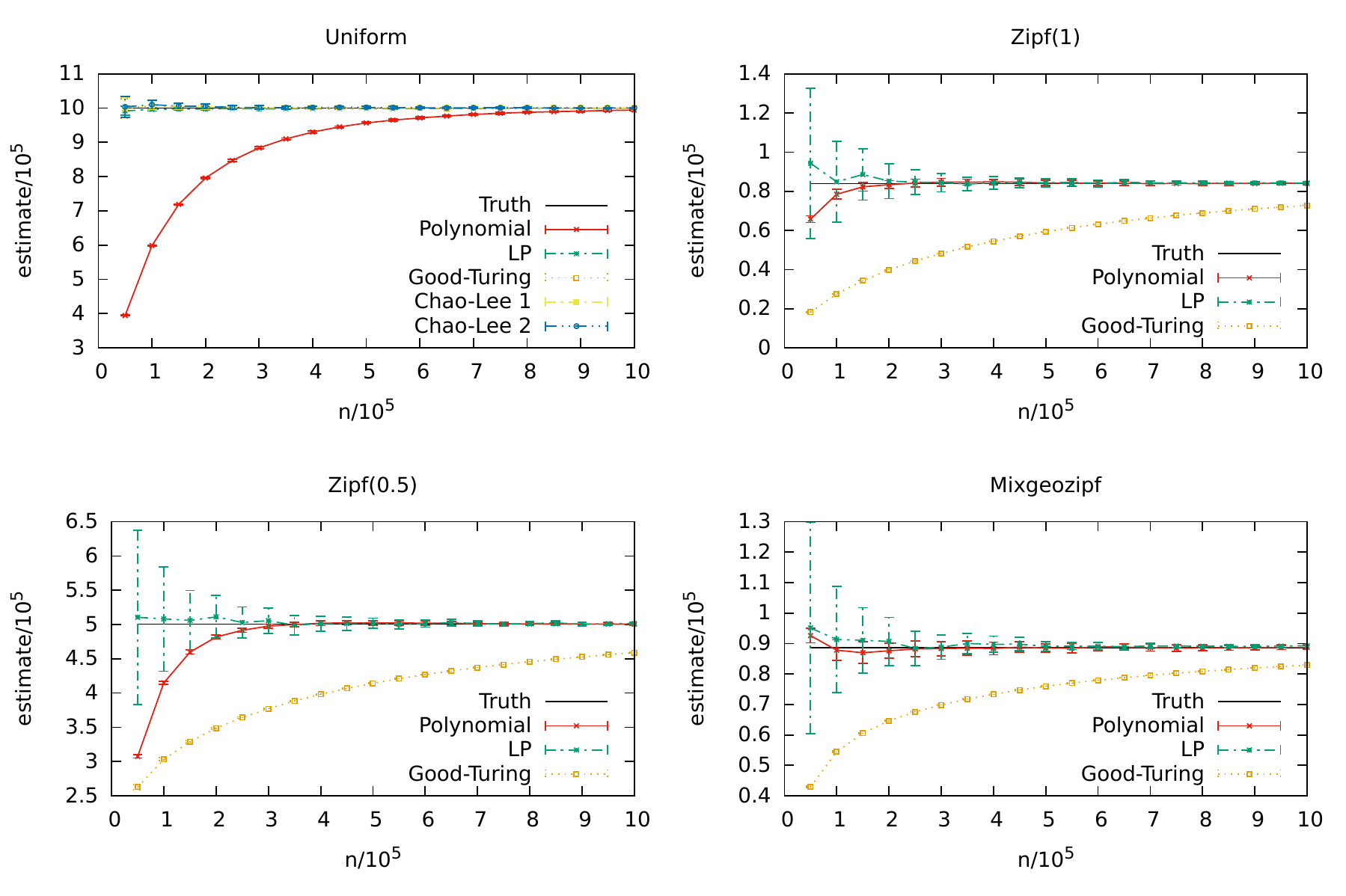}
    \caption{Performance comparison under four data-generating distributions.\label{fig:synthetic} }
\end{figure}
We compare our results with the Good-Turing estimator \cite{Good1953}, the two estimators proposed by Chao and Lee \cite{Chao92}, and the linear programming approach proposed by Valiant and Valiant \cite{VV13}.
Here the Good-Turing estimator refers to first estimate the total probability of seen symbols (sample coverage) by $\hat C=1- \frac{h_1}{n}$ then estimate the support size by $\hat S = {\Splug}/{\hat C}$.
The plug-in estimator simply counts the number of distinct elements observed, which is always outperformed by the Good-Turing estimator in our experiments and hence omitted in the comparison.

Good-Turing's estimate on sample coverage performs remarkably well in the special case of uniform distributions.
This has been noticed and analyzed in \cite{Chao92,DR80}.
Chao-Lee's estimators are based on Good-Turing's estimate with further correction terms for non-uniform distributions. 
However, with limited number of samples, if no symbol appears more than once, the sample coverage estimate $\hat C$ is zero and consequently the Good-Turing estimator and Chao-Lee estimators are not even well-defined.
For Zipf and mixture distributions, the output of Chao-Lee's estimators is highly unstable and thus is omitted from the plots;
the convergence rate of Good-Turing estimator is much slower than our estimator and the linear programming approach, partly because it only uses the information of how many symbols occurred exactly once, namely $h_1$, instead of the full spectrum of fingerprints $\{h_j\}_{j\geq 1}$;
the linear programming approach has similar convergence rate to ours but suffers large variance when samples are scarce.

\paragraph{Real data}
Next we evaluate our estimator by a real data experiment based on the text of \emph{Hamlet},
which contains about $ 32,000 $ words in total consisting of about $ 4,800 $ distinct words.
Here and below the definition of ``distinct word'' is any distinguishable arrangement of letters that are delimited by spaces, insensitive to cases, with punctuations removed.
We randomly sample the text with replacement and generate the fingerprints for estimation.
The minimum non-zero mass is naturally the reciprocal of the total number of words, $ \frac{1}{32,000} $.
In this experiment we use the degree-$ 4 $ Chebyshev polynomial.
We also compare our estimator with the one in \cite{VV13}.
The results are plotted in \prettyref{fig:estimation},
\begin{figure}[ht]
    \centering
    \includegraphics[width=0.9\linewidth]{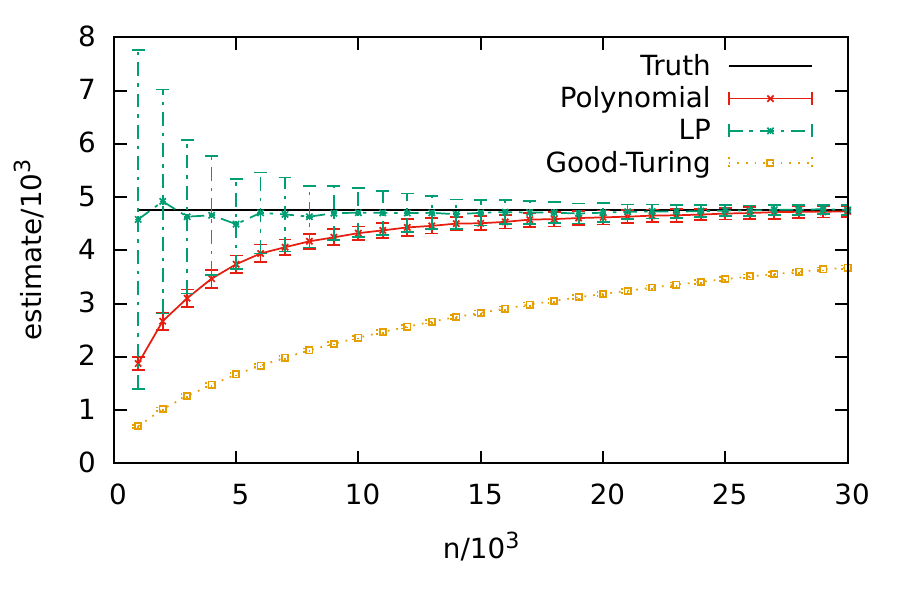}
    \caption{Comparison of various estimates of the total number of distinct words in \emph{Hamlet}. \label{fig:estimation} }
\end{figure}
which shows that the estimator in \cite{VV13} has similar convergence rate to ours; however, the variance is again much larger and the computational cost of linear programming is significantly higher than linear estimators, which amounts to computing linear combinations with pre-determined coefficients.
% The plug-in estimation is the number of observed words in the samples, which is far from the truth due to the amount of words never yet observed. 

On a larger scale experiment we used the \emph{New York Times Corpus} from the years 1987 -- 2007.\footnote{Data available at \url{https://catalog.ldc.upenn.edu/LDC2008T19}.}
%\footnote{Data are obtained from \url{https://cogcomp.cs.illinois.edu/page/corpora_view/112}.}.
This corpus has a total of 25,020,626 paragraphs consisting of 996,640,544 words with 2,047,985 distinct words.
% according the same definition of ``distinct'' as before.
We randomly sample 1\% -- 50\% out of the all paragraphs with replacements and feed the fingerprint to our estimator.
The minimum non-zero mass is also the reciprocal of the total number of words, $ 1/10^9 $, and thus the degree-9 Chebyshev polynomial is applied.
\begin{figure}[ht]
    \centering
    \includegraphics[width=0.9\linewidth]{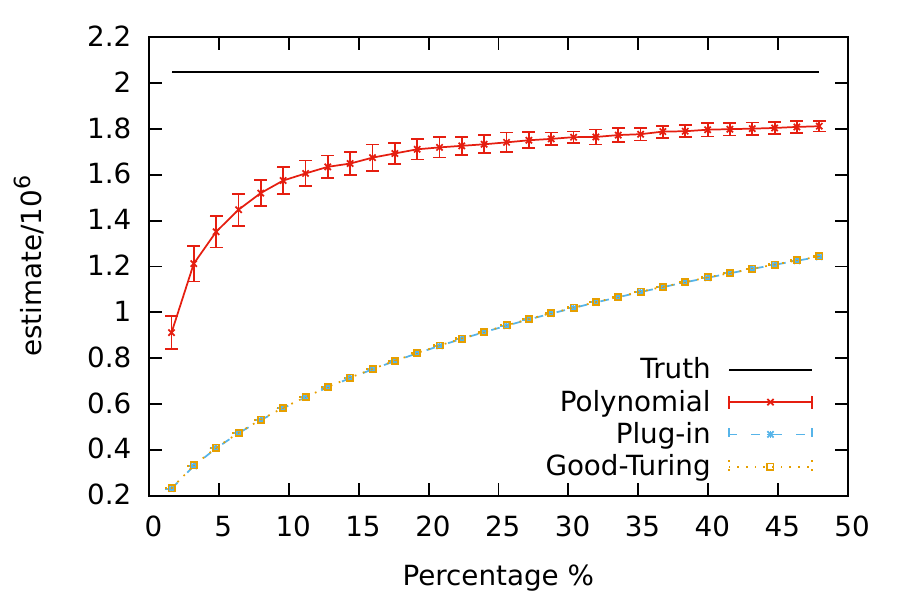}
    \caption{Performance of our estimator using \emph{New York Times Corpus}. \label{fig:NYT} }
\end{figure}
% This sampling procedure does not exactly match the word-based \iid sampling model.
Using only 20\% samples our estimator achieves a relative error of about 10\%, which is a systematic error due to the model mismatch: the sampling here is paragraph by paragraph rather than word by word, which induces dependence across samples as opposed to the iid sampling model for which the estimator is designed.
For this large dataset the linear programming estimator has unbearable computational cost: Even for the data of a single year the linear programming takes over 100 hours to compute on a server with E5-2623 CPU and 96 GB RAM; in contrast, the proposed linear estimator takes less than 15 minutes to run for the entire 20-year dataset on the same computer, which clearly demonstrates its computational advantage even if one factors into the difference that our implementation is based on C++ instead of MATLAB used in \cite{VV13}.

Finally, we perform the classical experiment of ``how many words did Shakespeare know''. We feed the fingerprint of the entire Shakespearean canon (see \cite[Table 1]{ET76}), which contains 31,534 word types, to our estimator.
We choose the minimum non-zero mass to be the reciprocal of the total number of English words, which, according to known estimates, is between 600,000 \cite{OED} to 1,000,000 \cite{GLM}, and obtain an estimate of 63,148 to 73,460 for Shakespeare's vocabulary size, as compared to 66,534 obtained by Efron-Thisted \cite{ET76}.

% \appendices
\section{Proof of lower bounds}
\label{sec:pf-lb}

\subsection{Proof of \prettyref{prop:main-lb2}}
%       \label{sec:}
\begin{proof}
    For $0<\nu<1$, define the set of \emph{approximate} probability vectors by
    \begin{equation*}
        \calD_k(\nu) \triangleq \sth{P=(p_1,p_2,\dots): \abs{\sum_{i} p_i-1}\le \nu, p_i\in \sth{0}\cup \qth{\frac{1+\nu}{k},1} }.    
    \end{equation*}
    which reduces to the original probability distribution space $\calD_k$ if $\nu=0$.
    Generalizing the sample complexity $ n^*(k,\epsilon) $ in \prettyref{eq:nstar} to the Poisson sampling model over $ \calD_k(\nu) $, we define 
    \begin{equation}
        n^*(k,\epsilon,\nu) \triangleq \min\{n \geq 0\colon \exists \hat S, \text{ s.t. } \Prob[|\hat S - S(P)| \ge \epsilon k] \leq 0.1,  \forall P\in \calD_k(\nu) \},
        \label{eq:nstar-tilde}
    \end{equation}
    where  $\hat S$ is an integer-valued estimator measurable with respect to $ N=(N_1,N_2,\dots)\inddistr \Poi(np_i) $.    
    % $ \tilde{n}^*(k,\epsilon,\nu) $ as the minimal sample size $n$ such that there exists an integer-valued estimator $\hat S$  measurable to $ N=(N_1,N_2,\dots)\inddistr \Poi(np_i) $, such that $\Prob[|\hat S - S(P)| \ge \epsilon k] \leq 0.2$ for any $P\in \calD_k(\nu)$, 
    The sample complexity of the fixed-sample-size and Poissonized model is related 
    % which is connected to $ n^*(k,\epsilon) $ 
    by the following lemma:
    \begin{lemma}
        \label{lmm:lb-generalize}
        For any $ \nu\in(0,1) $ and any $\epsilon \in (0,\frac{1}{2})$, 
        \begin{equation}
            % n^*(k,\epsilon)\ge (1-\nu)\tilde{n}^*(k,\epsilon,\nu) \pth{1- \frac{C}{\sqrt{\tilde{n}^*(k,\epsilon,\nu)}}},
            n^*(k,\epsilon)\ge (1-\nu)\tilde{n}^*(k,\epsilon,\nu) \pth{1- O\pth{\frac{1}{\sqrt{(1-\nu)\tilde{n}^*(k,\epsilon,\nu)}}}}.
            \label{eq:lb-generalize}
        \end{equation}
        % where $C$ is an absolute constant.
    \end{lemma}

    To establish a lower bound of $ \tilde{n}^*(k,\epsilon,\nu) $, we apply generalized Le Cam's method involving two composite hypothesis.
    Given two random variables $ U,U'\in[0,k] $ with unit mean
    we can construct two random vectors by $ \sfP=  \frac{1}{k} (U_1,\ldots,U_k) $ and $ \sfP'=  \frac{1}{k} (U_1',\ldots,U_k') $ with \iid~entries.
    Then $ \Expect[S(\sfP)]-\Expect[S(\sfP')]= k(\Prob[U>0]-\Prob[U'>0]) $.
    Furthermore, both $S(\sfP)$ and $S(\sfP')$ are binomially distributed, which are tightly concentrated at the respective means. 
    % From the law of large numbers that the functional values are concentrated at the corresponding means. 
    We can reduce the problem to the separation on mean values, as shown in the next lemma:
    \begin{lemma}
        \label{lmm:lb-poisson}
        Let $ U,U'\in\sth{0}\cup[1+\nu,\lambda] $ be random variables such that $ \Expect[U]=\Expect[U']=1 $, $ \Expect[U^j]=\Expect[U'^j] $ for $ j\in [L] $, and $ |\Prob[U>0]-\Prob[U'>0]|= d $.
        Then, for any $ \alpha<1/2 $, 
        \begin{equation}
            \frac{2\lambda}{k\nu^2}+\frac{2}{k\alpha^2d^2} + k\pth{\frac{en\lambda}{2kL}}^L\le 0.6
            \Rightarrow 
            \tilde{n}^*\pth{k,\frac{(1-2\alpha)d}{2},\nu} \ge n.
            % \tilde{R}^*(k,n,\nu)
            % \ge \frac{(1-2\alpha)^2k^2d^2}{4}\pth{1-\frac{2\lambda^2}{k\nu^2}-\frac{2}{k\alpha^2d^2} - k\pth{\frac{en\lambda}{2kL}}^L }.
            \label{eq:lb-poisson}
        \end{equation}
        % \begin{equation}
        %     \tilde{n}^*(k,d/2,\nu)
        %     \ge ...
        %     \label{eq:lb-poisson}
        % \end{equation}
    \end{lemma}

    Applying \prettyref{lmm:FL} in \prettyref{app:opt-UX}, we obtain two random variables $ U,U'\in \sth{0}\cup [1+\nu,\lambda] $ such that $ \Expect[U]=\Expect[U']=1 $, $ \Expect[U^j]=\Expect[U'^j],j=1,\dots,L $ and 
    \begin{equation*}
        \Prob[U>0]-\Prob[U'>0]=2E_{L-1}\pth{\frac{1}{x},[1+\nu,\lambda]}
        % =\frac{2}{\lambda-(1+\nu)}E_{L-1}\pth{\frac{1}{x-\frac{\lambda+(1+\nu)}{\lambda-(1+\nu)}},[-1,1]}
        % =\frac{(\frac{1}{\sqrt{\lambda}}+\frac{1}{\sqrt{1+\nu}})^2}{2}\pth{\frac{\sqrt{\lambda}-\sqrt{1+\nu}}{\sqrt{\lambda}+\sqrt{1+\nu}}}^{L}.
        =\frac{\pth{1+\sqrt{\frac{1+\nu}{\lambda}}}^2}{1+\nu}\pth{1-\frac{2\sqrt{\frac{1+\nu}{\lambda}}}{1+\sqrt{\frac{1+\nu}{\lambda}}}}^{L}\triangleq  d,
    \end{equation*}
    where the value of $ E_{L-1}(\frac{1}{x},[1+\nu,\lambda]) $ follows from \cite[2.11.1]{timan63}.
    To apply \prettyref{lmm:lb-poisson} and obtain a lower bound of $ \tilde{n}^*(k,\epsilon,\nu) $, we need to pick the parameters depending on the given $ k $ and $ \epsilon $ to fulfill:
    \begin{align}
      &\frac{(1-2\alpha)d}{2}\ge \epsilon,\label{eq:separation-ref}\\
      &\frac{2\lambda}{k\nu^2}+\frac{2}{k\alpha^2d^2} + k\pth{\frac{en\lambda}{2kL}}^L \le 0.6\label{eq:tv-ref}.
    \end{align}
    
    Let
    \begin{align*}
      L=\floor{c_0\log k},\quad &\lambda=\pth{\frac{\gamma \log k}{\log(1/2\epsilon)}}^2,\quad n=C\frac{k}{\log k}\log^2\frac{1}{2\epsilon},\\
      \alpha=\frac{1}{k^{1/3}}, \quad &\nu=\sqrt{\sqrt{\lambda/k}(1-2\epsilon)},
    \end{align*}
    for some $ c_0, \gamma , C \asymp 1 $ to be specified, and by assumption $ L,\lambda\rightarrow\infty $, $ \frac{\alpha}{1-2\epsilon}=o_k(1) $, $ \frac{\nu}{1-2\epsilon}=o_\tau(1)+o_k(1) $, $ 1/\lambda=o_\delta(1) $ .
    Since $ d\ge \frac{1}{1+\nu}(1-2\sqrt{\frac{1+\nu}{\lambda}})^L $, a sufficient condition for \prettyref{eq:separation-ref} is that
    \begin{equation}
        \pth{1-2\sqrt{\frac{1+\nu}{\lambda}}}^{L}\ge 2\epsilon\frac{1+\nu}{1-2\alpha}
        \Leftrightarrow
        % \frac{2(1+o_\nu(1)+o_\delta(1))\log(1/2\epsilon)}{\gamma \log k}\le \frac{\log(1/2\epsilon)(1+o_k(1)+o_\tau(1))}{c_0}}.
        \frac{\gamma}{c_0}>2+o_\tau(1)+o_\delta(1)+o_k(1).
        \label{eq:separation-ref2}
    \end{equation}
    % If we pick $ \nu,\alpha\ll 1-2\epsilon\le 1 $, then a further sufficient condition for \prettyref{eq:separation-ref} is that
    % \begin{equation*}
    %     \frac{2(1+o_\nu(1)+o_\lambda(1))}{\gamma}<\frac{1+o_{\nu/(1-2\epsilon)}(1)+o_{\alpha/(1-2\epsilon)}(1)}{c_0+o_k(1)}.
    % \end{equation*}
    
    Now we consider \prettyref{eq:tv-ref}.
    By the choice of $ \nu $ and $ \alpha $, we have
    % For the first two terms in \prettyref{eq:tv-ref} to vanish, we can pick
    \begin{equation*}
        \nu\gg \sqrt{\lambda/k},\quad \alpha\gg 1/\sqrt{k}d,
    \end{equation*}
    since $ 1-2\epsilon\gg \frac{\sqrt{\log k}}{k^{1/4}} $, $ d\ge \frac{2\epsilon}{1-2\alpha} $ and $ \epsilon=k^{-o(1)} $.
    % , and $ 1-2\epsilon\gg \frac{\sqrt{\log k}}{k^{1/4}} $, then $ \alpha=k^{-1/3} $ suffices; 
    % Since $ \lambda\asymp \frac{\log^2k}{\log^2(1/2\epsilon)} $ and $ 1-2\epsilon\gg \frac{\sqrt{\log k}}{k^{1/4}} $, then $ \nu=\sqrt{\sqrt{\lambda/k}(1-2\epsilon)} $ suffices.
    % Recall that $ \delta= \frac{\log\frac{1}{\epsilon}}{\log k}$ and $ \tau= \frac{\sqrt{\log k}/k^{1/4}}{1-2\epsilon} $
    % By those choices, the condition in \prettyref{eq:separation-ref2} is equivalent to
    % \begin{equation}
    %     \gamma>2c_0(1+o_\delta(1)+o_\tau(1)+o_k(1)),
    %     \label{eq:const-ref}
    % \end{equation}
    Then the first two terms in \prettyref{eq:tv-ref} vanish.
    % For the last term in \prettyref{eq:tv-ref} to vanish, we can pick 
    % \begin{equation*}
    %     n< C\frac{k}{\log k}\log^2\frac{1}{2\epsilon},
    % \end{equation*}
    The last term in \prettyref{eq:tv-ref} vanishes as long as the constant $ C<\frac{2c_0}{e\gamma^2}e^{-1/c_0} $.
    By the fact that
    \begin{equation*}
        \sup\sth{\frac{2c_0}{e\gamma^2}e^{-1/c_0}:0< 2c_0<\gamma}=\frac{1}{2e^2}, 
    \end{equation*}
    the optimal $ C $ satisfying \prettyref{eq:separation-ref2} is $ \frac{1+o_\delta(1)+o_\tau(1)+o_k(1)}{2e^2} $.
    Therefore, combining \prettyref{eq:separation-ref} -- \prettyref{eq:tv-ref} and applying \prettyref{eq:lb-poisson}, we obtain a lower bound of $ \tilde{n}^* $ that
    \begin{equation*}
        \tilde{n}^*(k,\epsilon,\nu)\ge \frac{1+o_\delta(1)+o_\tau(1)+o_k(1)}{2e^2}\frac{k}{\log k}\log^2\frac{1}{2\epsilon}.
    \end{equation*}
    Since $ 1-2\epsilon\gg \frac{\sqrt{\log k}}{k^{1/4}} $, we have $ \tilde{n}^*(k,\epsilon,\nu)\gg \sqrt{k} $.
    Applying \prettyref{lmm:lb-generalize}, we conclude the desired lower bound of $ n^*(k,\epsilon) $.    
\end{proof}

\subsection{Lower bound parts of Theorems \ref{thm:main} and \ref{thm:sample}}
\label{sec:pf-thm-lb}
\begin{proof}[Proof of lower bound of \prettyref{thm:sample}]
    The lower bound part of \prettyref{eq:sample-complexity-asymp} follows from \prettyref{prop:main-lb2}.
    Consequently, we obtain the lower bound part of \prettyref{eq:sample-complexity} for $ \frac{1}{k^c}\le\epsilon\le c_0 $ for the fixed constant $ c_0<1/2 $.
    
    The lower bound part of \prettyref{eq:sample-complexity} for $ \frac{1}{k}\le\epsilon\le \frac{1}{k^c} $ simply follows from the fact that $ \epsilon\mapsto n^*(k,\epsilon) $ is decreasing: 
    \begin{equation*}
        n^*(k,\epsilon)
        \ge n^*(k,1/k^c)
        \gtrsim k \log k
        \asymp \frac{k}{\log k}\log^2\frac{1}{\epsilon}.\qedhere
    \end{equation*}
\end{proof}

\begin{proof}[Proof of lower bound of \prettyref{thm:main}]
    % Now we consider the lower bound of the minimax quadratic risk in \prettyref{thm:main}: 
    By the Markov inequality,
    \begin{equation*}
        n^*(k,\epsilon)>n \Rightarrow R^*(k,n)>0.1\epsilon^2.
    \end{equation*}
    Therefore, our lower bound is
    \begin{equation*}
        R^*(k,n)\ge \sup\{0.1\epsilon^2:n^*(k,\epsilon)>n\}= 0.1\epsilon_*^{2},
    \end{equation*}
    where $ \epsilon_*\triangleq \{\epsilon:n^*(k,\epsilon)>n\} $.
    By the lower bound of $ n^*(k,\epsilon) $ in \prettyref{eq:main-lb2}, we obtain that
    \begin{equation*}
        \epsilon_*\ge \exp\pth{-\pth{\sqrt{2}e+o_\delta(1)+o_{\delta'}(1)+o_k(1)}\sqrt{\frac{n\log k}{k}}},
    \end{equation*}
    as $ \delta\triangleq \frac{n}{k\log k}\rightarrow 0 $, $ \delta'\triangleq \frac{k}{n\log k}\rightarrow 0 $, and $ k\diverge $.
    % where $ n\ll k\log k $ implies that $ \tilde{\epsilon}_*=k^{-o(1)} $, and $ n\gg \frac{k}{\log k} $ implies that $ \epsilon=o(1) $.
    Then we conclude the lower bound part of \prettyref{eq:main-asymp}, which implies the lower bound part of \prettyref{eq:main} when $ n\lesssim k\log k $.
    % By picking $ \epsilon=\exp\pth{-\pth{\sqrt{2}e+o(1)}\sqrt{\frac{n\log k}{k}}} $,  yields that of \prettyref{thm:main} in the regime of $ n\lesssim k\log k $.
    % It is impossible to distinguish two support sizes that differ by less than one since support sizes are integers, so \prettyref{thm:sample} does not apply when $ n\gtrsim k\log k $ such that $ k\epsilon<1 $

    For the lower bound part of \prettyref{eq:main} when $ n\gtrsim k\log k $, we apply Le Cam's two-point method \cite{Lecam86} by considering two possible distributions, namely $ P=\Bern(0)$ and $ Q=\Bern(\frac{1}{k})$.
    Then 
    \begin{equation*}
        R^*(k,n)\ge \frac{1}{4}(S(P)-S(Q))^2\exp(-nD(P\|Q))
        = \frac{k^2}{4}\exp\pth{n\log\pth{1-\frac{1}{k}}-2\log k}.
    \end{equation*}
    Since $ n\gtrsim k\log k $, we have $ n\log\pth{1-\frac{1}{k}}-2\log k\gtrsim -\frac{n}{k} $.
\end{proof}

\subsection{Proof of lemmas}
    \begin{proof}[Proof of \prettyref{lmm:lb-generalize}]
        Fix an arbitrary $P=(p_1,p_2,\ldots)\in \calD_k(\nu)$.
        Let $N=(N_1,N_2,\ldots)\inddistr \Poi(n p_i)$ and let $n' = \sum N_i \sim \Poi(n \sum p_i)\ge_{\text s.t.} \Poi(n(1-\nu))$.
        Let $\hat{S}_n$ be the optimal estimator of support size for fixed sample size $n$, such that whenever $ n\ge n^*(k,\epsilon) $ we have $ \Prob[|\hat S_n-S(P)|\ge \epsilon k]\le 0.1 $ for any $ P\in \calD_k $.
        We construct an estimator for the Poisson sampling model by $ \tilde{S}(N)= \hat{S}_{n'}(N). $ 
        We observe that conditioned on $ n'=m $, $ N\sim \Multinom( m, \frac{P}{\sum_i p_i} ) $.
        Note that $ \frac{P}{\sum_i p_i}\in\calD_k $ by the definition of $ \calD_k(\nu) $.
        Therefore 
        \begin{align*}
          \prob{ \abs{\tilde{S}(N)-S(P)}\ge \epsilon k }
          & = \sum_{m=0}^{\infty} \prob{ \abs{\hat{S}_m(N)-S\pth{ \frac{P}{\sum_ip_i} } }\ge \epsilon k } \prob{n'=m} \\ 
          &\le 0.1 \,\Prob[n'\ge n^*] +  \Prob[n'< n^*]
            =0.1 + 0.9\,\Prob[n'< n^*] \\
          &\le  0.1 + 0.9\,\Prob[\Poi(n(1-\nu))<n^*]. 
        \end{align*}
        If $ n= \frac{1+\beta}{1-\nu}n^* $ for $ \beta>0 $, then Chernoff bound (see, \eg, \cite[Theorem 5.4]{MU06}) yields that
        \begin{equation*}
            \Prob[\Poi(n(1-\nu))<n^*]
            \le \exp(-n^*(\beta-\log (1+\beta))).
        \end{equation*}
        By picking $ \beta=\frac{C}{\sqrt{n^*}} $ for some absolute constant $ C $, we obtain $ \tilde{n}^*\le \frac{n^*+C\sqrt{n^*}}{1-\nu} $ and hence
        the lemma.
        % If $ \rho<0.1 $ then $ \tilde n^*(k,\epsilon,\rho)\lesssim n^*(k,\epsilon) $.       
    \end{proof}

   \begin{proof}[Proof of \prettyref{lmm:lb-poisson}]
        Define two random vectors 
        \begin{equation*}
            \sfP=\pth{ \frac{U_1}{k},\dots,\frac{U_k}{k} } ,
            \quad \sfP'=\pth{ \frac{U_1'}{k},\dots,\frac{U_k'}{k} },
        \end{equation*}
        where $ U_i$ and $U_i' $ are \iid copies of $ U$ and $U' $, respectively.
        Conditioned on $ \sfP $ and $ \sfP' $ respectively, the corresponding histogram $N=(N_1,\ldots,N_k)\inddistr \Poi(n U_i/k)$ and $N' =(N'_1,\ldots,N'_k)\inddistr \Poi(n U_i'/k)$.
        % Denote by $N=(N_1,\ldots,N_k)$ and $N'=(N'_1,\ldots,N'_k)$ the corresponding histogram, which satisfy $N_i \inddistr \Poi(n U_i/k)$ and $N_i' \inddistr \Poi(n U_i'/k)$ conditioned on $\sfP$ and $\sfP'$, respectively.
        Define the following high-probability events: for $ \alpha<1/2 $,
        \begin{align*}
            E\triangleq & \sth{\abs{\frac{\sum_iU_i}{k}-1}\le \nu, \abs{S(\sfP)-\expect{S(\sfP)}}\le \alpha kd } , \\
            E'\triangleq & \sth{\abs{\frac{\sum_iU_i'}{k}-1}\le \nu, \abs{S(\sfP')-\expect{S(\sfP')}}\le \alpha kd } .
        \end{align*}
        Now we define two priors on the set $\calD_k(\nu)$ by the following conditional distributions:
        \begin{equation*}
            \pi = P_{\sfP|E},\quad\pi'=P_{\sfP'|E'}.
        \end{equation*}

        First we consider the separation of the support sizes under $ \pi$ and $\pi' $.
        Note that $ \Expect[S(\sfP)]=k\Prob[U>0] $ and $ \Expect[S(\sfP')]=k\Prob[U'>0] $, so $ |\Expect[S(\sfP)]-\Expect[S(\sfP')]|\ge kd $.
        By the definition of the events $ E,E' $ and the triangle inequality, we obtain that under $ \pi$ and $\pi' $, both $ \sfP,\sfP'\in \calD_k(\nu) $ and
        \begin{equation}
            |S(\sfP)-S(\sfP')|\ge (1-2\alpha)kd.
            \label{eq:lb-poisson-sep}
        \end{equation}
        
        Now we consider the total variation distance of the distributions of the histogram under the priors $ \pi$ and $\pi' $. 
        % Recall that conditioned on the values of $p_i$, $ N_i\sim \Poi(np_i) $. 
        By the triangle inequality and the fact that total variation of product distribution can be upper bounded by the summation of individual one,
        \begin{align}
          \TV ( P_{N|E}, P_{N'|E'} )
          & \le \TV ( P_{N|E}, P_N ) + \TV ( P_{N}, P_{N'} ) + \TV ( P_{N'} , P_{N'|E'} ) \nonumber\\
          & = \Prob[E^c] + \TV \pth{(\Expect[\Poi(nU/k)])^{\otimes k}, (\Expect[\Poi(nU'/k)])^{\otimes k} } + \Prob[E'^c]\nonumber\\
          & \le \Prob[E^c] + \Prob[E'^c] + k\TV(\Expect[\Poi(nU/k)], \Expect[\Poi(nU'/k)]).
            \label{eq:tv-ref1}
        \end{align}
        % Note that $ \epsilon= ... $.
        By the Chebyshev's inequality and the union bound, both
        \begin{align}
          \Prob[E^c],\Prob[E'^c]
          &\le \prob{\abs{\sum_i\frac{U_i}{k}-1}> \nu}+\prob{\abs{S(\sfP)-\expect{S(\sfP)}}> \alpha kd}\nonumber\\
          &\le \frac{\sum_i\var[U_i]}{(k\nu)^2}+\frac{\sum_i\var[\indc{U_i>0}]}{(\alpha kd)^2}
            \le \frac{\lambda}{k\nu^2}+\frac{1}{k\alpha^2d^2},
            \label{eq:tv-ref2}
        \end{align}
        where we upper bounded the variance of $ U $ by $ \var[U]\le \Expect[U^2]\le \Expect[\lambda U]=\lambda $.

        Applying the total variation bound for Poisson mixtures in \prettyref{lmm:tv-bound} (see \prettyref{app:tv-bound})  yields that
        \begin{equation}
            \TV(\Expect[\Poi(nU/k)], \Expect[\Poi(nU'/k)])
            \le \pth{\frac{en\lambda}{2kL}}^L.
            \label{eq:tv-ref3}
        \end{equation}
        Plugging \prettyref{eq:tv-ref2} and \prettyref{eq:tv-ref3} into \prettyref{eq:tv-ref1}, we obtain that
        \begin{equation}
            \TV ( P_{N|E}, P_{N'|E'} )
            \le \frac{2\lambda}{k\nu^2}+\frac{2}{k\alpha^2d^2} + k\pth{\frac{en\lambda}{2kL}}^L.
            \label{eq:lb-poisson-tv}
        \end{equation}
        Applying Le Cam's lemma \cite{Lecam86}, the conclusion follows from \prettyref{eq:lb-poisson-sep} and \prettyref{eq:lb-poisson-tv}.
    \end{proof}

\section{Proof of upper bounds}
\label{sec:pf-ub}
\subsection{Proof of Propositions \ref{prop:main-ub-plug} and \ref{prop:main-ub-poly}}

\begin{proof}[Proof of \prettyref{prop:main-ub-plug}]
    % Fix a distribution $ P $. 
    First we consider the bias:
    \begin{align*}
      |\Expect(\Splug(P)-S(P))|
      = &~\sum_{i}(1-\Prob(N_i\ge 1))\indc{p_i\ge \frac{1}{k}}
      = \sum_{i}\exp(-np_i)\indc{p_i\ge \frac{1}{k}}\\
      \le &~ k\exp(-n/k).
    \end{align*}
    The variance satisfies
    \begin{align*}
      \var [\Splug(P)]
      = \sum_{i}\var \indc{N_i>0} \indc{p_i\ge \frac{1}{k}}
      % = \sum_{i} \exp(-np_i)(1-\exp(-np_i))\indc{p_i\ge \frac{1}{k}}\\
      \le \sum_{i} \exp(-np_i)\indc{p_i\ge \frac{1}{k}} \le k\exp(-n/k).
    \end{align*}
    The conclusion follows.

    For the negative result, under the Poissonized model and with the samples drawn from the uniform distribution, the plug-in estimator $ \Splug$ is distributed as $\Binom(k,1-e^{-n/k}) $.
    If $ n\le (1-\delta)k\log\frac{1}{\epsilon}<k\log\frac{1}{\epsilon} $, then $ 1-e^{-n/k}<1-\epsilon $.
    By the Chernoff bound,
    \begin{align*}
        \Prob[|\Splug-S(P)|\le \epsilon k]
        = &~ \Prob[\Binom(k,1-e^{-n/k})\ge (1-\epsilon)k]\\
        \le &~ e^{-kd(1-\epsilon\|1-e^{-n/k})}
        = e^{-kd(\epsilon\|e^{-n/k})},
    \end{align*}
    where $ d(p\|q)\triangleq p\log\frac{p}{q}+(1-p)\log\frac{1-p}{1-q} $ is the binary divergence function.
    % For $ n\le \frac{k}{2}\log\frac{1}{\epsilon} $ and $ \epsilon\ge \frac{1}{k} $ we have
    Since $ e^{-n/k}\ge \epsilon^{1-\delta}> \epsilon $,
    \begin{equation*}
        d(\epsilon\|e^{-n/k})\ge d(\epsilon\|\epsilon^{1-\delta})\ge d(k^{-1}\|k^{-1+\delta})\asymp k^{-1+\delta},
    \end{equation*}
    where the middle inequality follows from the fact that $\epsilon \mapsto d(\epsilon\|\epsilon^{1-\delta})$ is increasing near zero.
    Therefore $ \Prob[|\Splug-S(P)|\le \epsilon k]\le \exp(-\Omega(k^\delta)) $.
\end{proof}

\begin{proof}[Proof of \prettyref{prop:main-ub-poly}]
    % Throughout the proof the notation $ o(1) $ can be replaced by a function $ f(\delta,k) $ satisfying $ f(0_+,\infty)=0 $ where $ \delta\triangleq \frac{l}{r}=\frac{n}{c_1k\log k} $, which implies the non-asymptotic statement.
    % Here we only prove the asymptotic statement.
    
    First we consider the bias. Recall that $    L =\floor{c_0\log k}, r=\frac{c_1\log k}{n}, l=\frac{1}{k}$.
    By \prettyref{eq:bias-term} the bias of $ \hat{S} $ is the summation of
    \begin{equation*}
        b(p_i) \triangleq e^{-np_i}P_L(p_i)\indc{p_i>0}.
    \end{equation*}
    Obviously $ b(0)=0 $.
    If $ l\le x\le r $ then $ |P_L(x)|\le \frac{1}{|T_L(-\frac{r+l}{r-l})|}=\frac{1}{|T_L(-\frac{1+\delta}{1-\delta})|} $ by the design of $ P_L $ in \prettyref{eq:PL}.
    Therefore $ |b(x)|\le e^{-nl}/|T_L(-\frac{1+\delta}{1-\delta})| $;
    if $ r<x\le 1 $,
    \begin{equation}
        |b(x)|\le \max_{r<x\le 1}e^{-nx}|P_L(x)|
        = \max_{1<y\le \frac{2-r-l}{r-l}}\exp(-nr(1-\delta)y/2)T_L(y)\frac{\exp(-nr(1+\delta)/2)}{|T_L(-\frac{1+\delta}{1-\delta})|}.
        \label{eq:bias-ref1}
    \end{equation}
    We need the following lemma:
%     (proved in \prettyref{app:max-exp-cheby}):
    \begin{lemma}
         If $\beta = O(L)$, then
        \begin{equation}
            \max_{x\ge 1}e^{-\beta x}T_L(x)=\frac{1}{2}\pth{\frac{\alpha+\sqrt{\alpha^2+1}}{e^{\sqrt{1+1/\alpha^2}}}(1+o_{L}(1))}^L, \quad L\diverge,
%            (1+o_{L\diverge}(1))}^L.
            \label{eq:max-exp-cheby}
        \end{equation}
        where $\alpha\triangleq \frac{L}{\beta}$.
        % In particular, \nb{non-asymptotic..}
        \label{lmm:max-exp-cheby}
    \end{lemma}
    Applying \prettyref{lmm:max-exp-cheby} to \prettyref{eq:bias-ref1} 
    with $ L=\floor{c_0\log k} $, $ \beta = nr(1-\delta)/2 $ and $ \alpha=2\rho+o(1) $ where $ \rho\triangleq c_0/c_1 $, we obtain that
    \begin{align*}
      |b(x)|&\le \frac{1}{2}\pth{\frac{2\rho+\sqrt{(2\rho)^2+1}}{e^{\sqrt{1+1/(2\rho)^2}}}(1+o_{k}(1))}^L\frac{\exp(-\frac{L}{2\rho} (1+o_\delta(1)))}{|T_L(-\frac{1+\delta}{1-\delta})|}\\
            &= \frac{1}{2}\pth{\frac{2\rho+\sqrt{(2\rho)^2+1}}{e^{\sqrt{1+1/(2\rho)^2}+1/(2\rho)}}(1+o_{k}(1)+o_\delta(1))}^L\frac{1}{|T_L(-\frac{1+\delta}{1-\delta})|}.
    \end{align*}
    Therefore $ b(p_i) $ is uniformly bounded by $ \frac{1+o_{k}(1)+o_\delta(1)}{|T_L(-\frac{1+\delta}{1-\delta})|} $ as long as we pick the constant $ \rho $ such that $ \frac{2\rho+\sqrt{(2\rho)^2+1}}{e^{\sqrt{1+1/(2\rho)^2}+1/(2\rho)}}<1 $, or equivalently, $ \rho<\rho^* \approx 1.1 $.
    Then the bias of $ \hat{S} $ is at most
    \begin{align}
      |\Expect[\hat{S}-S]|
      &\le k\frac{1+o_{k}(1)+o_\delta(1)}{|T_L(-\frac{1+\delta}{1-\delta})|}
      % = k\frac{2}{z^{L}+z^{-L}}
      % \le 2kz^{-L}
      \le 2k(1+o_{k}(1)+o_\delta(1))\pth{1-\frac{2\sqrt{\delta}}{1+\sqrt{\delta}}}^L\nonumber\\
      &= 2k(1+o_{k}(1))\exp\pth{-(1+o_\delta(1))\sqrt{4c_0\rho\frac{n\log k}{k}}}.
        \label{eq:main-ub-poly-bias}
    \end{align}
    % where $ z=\frac{1+\sqrt{\delta}}{1-\sqrt{\delta}} $. 

    Now we turn to the variance of $ \hat{S} $:
    \begin{align*}
      \var[\hat{S}]
      &=\sum_{i:p_i>0}\var\qth{(g_L(N_i)-1)\indc{N_i\le L}} \nonumber \\
      &\le \sum_{i:p_i>0}\Expect\qth{(g_L(N_i)-1)^2\indc{N_i\le L}} \\
        &=\sum_{i:p_i>0}\sum_{j=0}^{L}\pth{\frac{a_jj!}{n^j}}^2e^{-np_i}\frac{(np_i)^j}{j!} 
        = \sum_{j=0}^{L}\pth{\frac{a_jj!}{n^j}}^2\Expect[h_j],
    \end{align*}
    where $ h_j\triangleq \sum_i \indc{N_i=j} $ is the fingerprint of samples.
    By definition $ \Expect[\sum_j j h_j]=n $.
    Therefore
    \begin{align}
      \var[\hat{S}]
      \le \Expect[h_0]+\sum_{j=1}^{L}\frac{(a_jj!/n^j)^2}{j}\Expect[jh_j]
      \le k+nL\max_{1\le j\le L}\frac{(a_jj!/n^j)^2}{j}.
      \label{eq:var-coeff-ref0}
    \end{align}
    Recall the polynomial coefficients $ a_j $ given in \prettyref{eq:aj}:
    \begin{align*}
      |a_j|=\pth{\frac{2}{r-l}}^j\frac{1}{j!}\frac{|T_L^{(j)}(-\frac{r+l}{r-l})|}{|T_L(-\frac{r+l}{r-l})|}.
    \end{align*}
    Applying Markov brothers' inequality \cite{markov1892} on the scaled interval $ [-\frac{r+l}{r-l},\frac{r+l}{r-l}] $, we obtain that 
    \begin{align}
      \abs{\frac{j!}{n^j}a_j}
      \le \frac{j!}{n^j}\pth{\frac{2}{r-l}}^j\frac{1}{j!}\frac{2^jj!}{|\frac{r+l}{r-l}|^j}\frac{L}{L+j}\binom{L+j}{2j}
      = \pth{\frac{4}{n(r+l)}}^jj!\frac{L}{L+j}\binom{L+j}{2j}.
      \label{eq:var-coeff-ref1}
    \end{align}
    We use the following bound on binomial coefficients \cite[Lemma 4.7.1]{ash-itbook}:
    \begin{align}
      \frac{\sqrt{\pi}}{2} \leq \frac{\binom{n}{k}}{(2 \pi n \lambda (1-\lambda))^{-1/2}  \exp(n h(\lambda))} \leq 1  .
      \label{eq:ash2}
    \end{align}
    where $\lambda = \frac{k}{n} \in (0,1)$ and  $h(\lambda) \triangleq -\lambda \log \lambda - (1-\lambda) \log (1-\lambda)$ denotes the binary entropy function.
    Therefore, from \prettyref{eq:var-coeff-ref1} and \prettyref{eq:ash2}, for $ j=1,\dots,L-1 $,
    \begin{align}
      \abs{\frac{a_jj!}{n^j}}
      \le &~\pth{\frac{4}{n(r+l)}}^jj!\frac{L}{L+j}\frac{\exp((L+j)h(\frac{2j}{L+j}))}{\sqrt{2\pi\cdot 2j\frac{L-j}{L+j}}} \nonumber \\
      \le &~\pth{\frac{4}{nr}}^j\frac{j!}{2}\exp\pth{(L+j)h\pth{\frac{2j}{L+j}}},
      \label{eq:var-coeff-ref2}
    \end{align}
    where we used the fact that $ \max_{j\in [L-1]}\frac{L}{\sqrt{4\pi j(L-j)(L+j)}}= \frac{L}{\sqrt{4\pi(L^2-1)}}\le \frac{1}{2} $ for $ L\ge 2 $.
    From \prettyref{eq:var-coeff-ref1}, the upper bound \prettyref{eq:var-coeff-ref2} also holds for $ j=L $.
    Using \prettyref{eq:var-coeff-ref2} and Striling's approximation that $ n!<e\sqrt{n}(\frac{n}{e})^n $,
    \begin{align}
      \frac{(a_jj!/n^j)^2}{j}
      &\le \frac{1}{j}\pth{\frac{4}{c_1\log k}}^{2j}\pth{\frac{e\sqrt{j}}{2}}^2\pth{\frac{j}{e}}^{2j}\exp\pth{2(L+j)h\pth{\frac{2j}{L+j}}}\nonumber\\
      &= \frac{e^2}{4}k^{2c_0(\beta\log\frac{4\rho \beta}{e}+(1+\beta)h(\frac{2\beta}{1+\beta}))}
        \le \frac{e^2}{4}k^{2c_0\tau(\rho)},\label{eq:var-coeff-ref3}
        % \exp\pth{2L\pth{\beta\log\frac{4\rho \beta}{e}+(1+\beta)h\pth{\frac{2\beta}{1+\beta}}}}
    \end{align}
    where $ \beta\triangleq j/L $ and $ \tau(\rho)\triangleq \sup_{\beta\in[0,1]}(\beta\log\frac{4\rho \beta}{e}+(1+\beta)h(\frac{2\beta}{1+\beta})) $, which occurs at $ \beta=\frac{\sqrt{1+4\rho^2}-1}{2\rho} $.
    Note that from \prettyref{eq:main-ub-poly-bias} the squared bias of $ \hat{S} $ is $ 4k^{2-o_\delta(1)} $; from \prettyref{eq:var-coeff-ref0} and \prettyref{eq:var-coeff-ref3} the variance of $ \hat{S} $ is at most
    \begin{equation}
        \var[\hat{S}]\le k+\frac{e^2}{4}nLk^{2c_0\tau(\rho)},
        \label{eq:main-ub-poly-var}
    \end{equation}
    which is $ \frac{e^2}{4}k^{1+2c_0\tau(\rho)+o_\delta(1)} $.
    Therefore as long as we pick constant $ c_0 $ such that $ 2c_0\tau(\rho)<1 $ the variance of $ \hat{S} $ in \prettyref{eq:main-ub-poly-var} is lower order than the squared bias of $ \hat{S} $ in \prettyref{eq:main-ub-poly-bias}, and thus the MSE of $ \hat{S} $ is at most
    \begin{equation*}
        \Expect(\hat{S}-S)^2\le 4k^2(1+o_{k}(1))\exp\pth{-2(1+o_\delta(1))\sqrt{\frac{2\rho}{\tau(\rho)}\frac{n\log k}{k}}}.
    \end{equation*}
    The conclusion follows from the fact that $\sup_{\rho < \rho^*} 2 \rho/\tau(\rho) \approx 2.494 $, which corresponds to choosing
%    the maximal value of $ \rho/\tau(\rho) $ satisfying $ \rho<\rho^* $.
    $ c_0 \approx 0.558 $ and $ c_1=0.5 $.
%    \nbr{exact 0.5? coincide?}   
\end{proof}

\subsection{Upper bound parts of Theorems \ref{thm:main} and \ref{thm:sample}}
\label{sec:pf-thm-ub}
\begin{proof}[Proof of upper bound of \prettyref{thm:main}]
    % The risk of plug-in estimator yields the upper bound of the risk in \prettyref{thm:main} for $ n\gtrsim k\log k $.
    Combining \prettyref{lmm:ub-poisson} and \prettyref{prop:main-ub-poly} yields the upper bound part of \prettyref{eq:main-asymp}, which also implies the upper bound of \prettyref{eq:main} when $ n\lesssim k\log k $.
    The upper bound part of \prettyref{eq:main} when $ n\gtrsim k\log k $ follows from \prettyref{prop:main-ub-plug}.
\end{proof}

\begin{proof}[Proof of upper bound of \prettyref{thm:sample}]
    % Now we consider the upper bound of the sample complexity in \prettyref{thm:sample}.
    By the Markov inequality,
    \begin{equation}
        R^*(k,n)\le 0.1\epsilon^2 \Rightarrow n^*(k,\epsilon)\le n.
        \label{eq:markov}
    \end{equation}
    Therefore our upper bound is
    \begin{equation*}
        n^*(k,\epsilon)\le \inf\{n:R^*(k,n)\le 0.1\epsilon^2\}.
    \end{equation*}
    By the upper bound of $ R^*(k,n) $ in \prettyref{eq:main-ub-poly}, we obtain that
    \begin{equation*}
        n^*(k,\epsilon)\le \frac{1+o_{\delta'}(1)+o_\epsilon(1)+o_k(1)}{\kappa}\frac{k}{\log k}\log^2\frac{1}{\epsilon}
    \end{equation*}
    as $ \delta'\triangleq\frac{\log(1/\epsilon)}{\log k}\triangleq 0 $, $ \epsilon\rightarrow 0 $, and $ k\diverge $.
    Consequently, we obtain the upper bound part of \prettyref{eq:sample-complexity} when $ \frac{1}{k^{c}}\le \epsilon\le c_0 $ for the fixed constant $ c_0<1/2 $.
    
    The upper bound part of \prettyref{thm:sample} when $ \frac{1}{k}\le \epsilon\le \frac{1}{k^{c}} $ follows from the monotonicity of $ \epsilon\mapsto n^*(k,\epsilon) $ that
    \begin{equation*}
        n^*(k,\epsilon)
        \le n^*(k,1/k) \le 3 k\log k \asymp \frac{k}{\log k}\log^2\frac{1}{\epsilon},
    \end{equation*}
    where the middle inequality follows from \prettyref{prop:main-ub-plug} and \prettyref{eq:markov}.
    % where we used the conclusion of the coupon collector's problem \nbr{should be \prettyref{prop:main-ub-plug}} that $ n^*(k,1/k) \lesssim k\log k $. 
\end{proof}

\subsection{Proof of lemmas}
\begin{proof}[Proof of \prettyref{lmm:ub-poisson}]
    We follow the same idea as in \cite[Appendix A]{WY14} using the Bayesian risk as a lower bound of the minimax risk with a more refined application of the Chernoff bound.
    We express the risk under the Poisson sampling as a function of the original samples that
    \begin{equation*}
        \tilde{R}^*(k,(1-\beta)n)
        =\inf_{\{\hat{S}_m\}}\sup_{P\in\calD_k}\Expect[\ell(\hat{S}_{n'},S(P))],
    \end{equation*}
    where $ n'\sim\Poi((1-\beta)n) $.
    The Bayesian risk is a lower bound of the minimax risk:
    \begin{equation}
        \tilde{R}^*(k,(1-\beta)n)
        \ge \sup_\pi\inf_{\{\hat{S}_m\}}\Expect[\ell(\hat{S}_{n'},S(P))],
        \label{eq:bayesian-ref1}
    \end{equation}
    where $\pi$ is a prior over the parameter space $\calD_k$. For any sequence of estimators $ \{\hat{S}_m\} $,
    \begin{equation*}
        \Expect[\ell(\hat{S}_{n'},S)]
        =\sum_{m\ge 0}\Expect[\ell(\hat{S}_{m},S)]\Prob[n'=m]
        \ge \sum_{m=0}^{n}\Expect[\ell(\hat{S}_{m},S)]\Prob[n'=m].
    \end{equation*}
    Taking infimum of both sides, we obtain
    \begin{equation*}
        \inf_{\{\hat{S}_m\}}\Expect[\ell(\hat{S}_{n'},S)]
        \ge \inf_{\{\hat{S}_m\}}\sum_{m=0}^{n}\Expect[\ell(\hat{S}_{m},S)]\Prob[n'=m]
        = \sum_{m=0}^{n}\inf_{\hat{S}_m}\Expect[\ell(\hat{S}_{m},S)]\Prob[n'=m].
    \end{equation*}
    Note that for any fixed prior $ \pi $, the function $ m\mapsto\inf_{\hat S_m}\Expect[\ell(\hat{S}_{m},S)] $ is decreasing.
    Therefore
    \begin{align}
      \inf_{\{\hat{S}_m\}}\Expect[\ell(\hat{S}_{n'},S)]
      &\ge \inf_{\hat{S}_n}\Expect[\ell(\hat{S}_{n},S)]\Prob[n'\le n] \nonumber \\
      &  \ge \inf_{\hat{S}_n}\Expect[\ell(\hat{S}_{n},S)](1-\exp(n(\beta+\log(1-\beta))))\nonumber\\
      &\ge \inf_{\hat{S}_n}\Expect[\ell(\hat{S}_{n},S)](1-\exp(-n\beta^2/2)),\label{eq:bayesian-ref2}
    \end{align}
    where we used the Chernoff bound (see, \eg, \cite[Theorem 5.4]{MU06}) and the fact that $ \log(1-x)\le -x-x^2/2 $ for $x>0$.
    Taking supremum over $ \pi $ on both sides of \prettyref{eq:bayesian-ref2}, the conclusion follows from \prettyref{eq:bayesian-ref1} and the minimax theorem (cf.~\eg \cite[Theorem 46.5]{Strasser85}).
    % \nb{double check}
\end{proof}

%\section{Proof of \prettyref{lmm:max-exp-cheby}}
%\label{app:max-exp-cheby}
\begin{proof}[Proof of \prettyref{lmm:max-exp-cheby}]
    By assumption, $\alpha = \frac{L}{\beta}$ is strictly bounded away from zero.
    Let $ f(x)\triangleq e^{-\beta x} T_L(x)=e^{-\beta x}\cosh(L\arccosh(x)) $ when $ x\ge 1 $.
    % Using the definition \prettyref{eq:cheby} of Chebyshev polynomial and 
    By taking the derivative of $ f $, we obtain that $ f $ is decreasing if and only if 
    \[
        \frac{\tanh(L\arccosh(x))}{\sqrt{x^2-1}}=\frac{\tanh(Ly)}{\sinh(y)}<\frac{1}{\alpha},
    \]
    % Change a variable that 
    where $ x=\cosh(y) $.
    Let $ g(y)=\frac{\tanh(Ly)}{\sinh(y)} $.
    Note that $ g $ is strictly decreasing on $ \reals_+ $ with $ g(0)=L $ and $ g(\infty)=0 $.
    Therefore $ f $ attains its maximum at $x^*$ which is the unique solution of $ \frac{\tanh(L\arccosh(x))}{\sqrt{x^2-1}}=\frac{1}{\alpha} $.
    It is straightforward to verify that the solution satisfies $ x^*=\sqrt{1+\alpha^2}(1-o_{L}(1)) $ when $ \alpha $ is strictly bounded away from zero.
    % since $ x^*\ge \sqrt{1+\alpha^2}-o_{L}(1) $ and thus $ \tanh(L\arccosh(x^*))=1-o_{L}(1) $. \nbr{fix!!}
    Therefore the maximum value of $ f $ is
    \begin{equation*}
        e^{-\beta x^*}T_L(x^*)=e^{-L\sqrt{1+1/\alpha^{2}}(1-o_{L}(1))}\frac{1}{2}(z^L+z^{-L}),
    \end{equation*}
    where we used \prettyref{eq:cheby} and $ z=x^*+\sqrt{x^{*2}-1}=(\sqrt{1+\alpha^2}+\alpha)(1-o_{L}(1))$ is strictly bounded away from 1.
    This proves the lemma.
    % Therefore
    % \begin{equation*}
    %     e^{-\beta x^*}T_L(x^*)
    %     =\frac{1}{2}(1+o(1))\pth{\frac{z}{e^{\sqrt{1+1/\alpha^{2}}(1-o(1))}}}^L.
    %     =\frac{1}{2}(1+o(1))\pth{\frac{\sqrt{1+\alpha^2}+\alpha}{e^{\sqrt{1+1/\alpha^{2}}}}(1+o(1))}^L.
    % \end{equation*}
\end{proof}

\appendix
\section{Dual program of \eqref{EQ:FL}}
\label{app:opt-UX}
Define the following infinite-dimensional linear programming problem:
\begin{equation}
    \begin{aligned}
        % \calF_L(\lambda) \triangleq
        \calE_1^*\triangleq
        \sup & ~ \prob{U'=0}-\prob{U=0}  \\
        \text{s.t.}     
        & ~ \expect{U} = \expect{U'}=1 \\
        & ~ \expect{U^j} = \expect{U'^j}, \quad j = 1,\ldots,L+1, \\
        & ~ U,U' \in \sth{0}\cup I,
    \end{aligned}
    \label{eq:FL1}
\end{equation}
where $ I=[a,b] $ with $ b>a\ge 1 $. Then \eqref{EQ:FL} is a special case of \prettyref{eq:FL1} with $ I=[1,\lambda] $. 
\begin{lemma}
    $\calE_1^* =\inf_{p\in\calP_L}\sup_{x\in I}\abs{\frac{1}{x}-p(x)}$.
    \label{lmm:FL}
\end{lemma}
\begin{proof}
    We first show that \prettyref{EQ:FL} coincides with the following optimization problem:
    \begin{equation}
        \begin{aligned}
            % \calF_L(\lambda) \triangleq
            \calE_2^*\triangleq
            \sup & ~ \expect{\frac{1}{X}}-\expect{\frac{1}{X'}}  \\
            \text{s.t.}     
            & ~ \expect{X^j} = \expect{X'^j}, \quad j = 1,\ldots,L, \\
            & ~ X,X' \in I.
        \end{aligned}
        \label{eq:FL2}
    \end{equation}
    Given any feasible solution $ U,U' $ to \prettyref{EQ:FL}, construct $ X,X' $ with the following distributions:
    \begin{equation}
        \begin{aligned}
            &P_X(\diff x)=xP_{U}(\diff x),\\
            &P_{X'}(\diff x)=xP_{U'}(\diff x),
        \end{aligned}
        \label{eq:U-X}
    \end{equation}
    It is straightforward to verify that $ X,X' $ are feasible for \prettyref{eq:FL2} and
    \begin{equation*}
        \calE_2^*
        \ge \expect{\frac{1}{X}}-\expect{\frac{1}{X'}}
        =\prob{U'=0}-\prob{U=0}.
    \end{equation*}
    Therefore $ \calE_2^*\ge \calE_1^* $.

    On the other hand, given any feasible $ X,X' $ for \prettyref{eq:FL2}, construct $ U,U' $ with the distributions:
    \begin{equation}
        \begin{aligned}
            &P_U(\diff u)=\pth{1-\expect{\frac{1}{X}}}\delta_0(\diff u)+\frac{1}{u}P_{X}(\diff u),\\
            &P_{U'}(\diff u)=\pth{1-\expect{\frac{1}{X'}}}\delta_0(\diff u)+\frac{1}{u}P_{X'}(\diff u),
        \end{aligned}
        \label{eq:X-U}
    \end{equation}
    which are well-defined since $ X,X'\ge 1 $ and hence $ \expect{\frac{1}{X}}\le 1, \expect{\frac{1}{X'}}\le 1 $.
    Then $ U,U' $ are feasible for \prettyref{EQ:FL} and hence 
    \begin{equation*}
        \calE_1^*
        \ge \prob{U'=0}-\prob{U=0}
        =\expect{\frac{1}{X}}-\expect{\frac{1}{X'}}.
    \end{equation*}
    Therefore $ \calE_1^*\ge \calE_2^* $. Finally, the dual of \prettyref{eq:FL2} is precisely the best polynomial approximation problem (see, \eg, \cite[Appendix E]{WY14}) and hence
    \begin{equation*}
        \calE_1^*=\calE_2^*=\inf_{p\in\calP_L}\sup_{x\in I}\abs{\frac{1}{x}-p(x)} . \qedhere
%        \label{eq:FL3}
    \end{equation*}
\end{proof}

\section{Total variation between Poisson mixtures}
\label{app:tv-bound}    

        The total variation distance between two Poisson mixtures is obtained in the following lemma, which is an improvement of \cite[Lemma 3]{WY14} in terms of constants. This is crucial for our purposes of obtaining the best constants in the sample complexity bounds in \prettyref{eq:sample-complexity-asymp}.
        \begin{lemma}
            \label{lmm:tv-bound}
            Let $ V $ and $ V' $ be random variables taking values on $ [0,\Lambda] $.
            If $ \Expect[V^j]=\Expect[V'^j],~j=1,\dots,L $, then
            \begin{equation}
                \TV(\Expect[\Poi(V)],\Expect[\Poi(V')])
                \le \frac{(\Lambda/2)^{L+1}}{(L+1)!}\pth{2+2^{\Lambda/2-L}+2^{\Lambda/(2\log 2)-L}}.
                % = {\sqrt{\frac{2}{\pi (L+1)}}}\pth{\frac{e\Lambda}{2(L+1)}}^{L+1}(1+o(1)).
                \label{eq:tv-bound}
            \end{equation}
            In particular, $ \TV(\Expect[\Poi(V)],\Expect[\Poi(V')])\le (\frac{e\Lambda}{2L})^{L} $. Moreover, if $ L>\frac{e}{2}\Lambda $, then
            \[
             \TV(\Expect[\Poi(V)],\Expect[\Poi(V')])\le \frac{2(\Lambda/2)^{L+1}}{(L+1)!}(1+o(1)), \quad \Lambda \diverge.
             \]
        \end{lemma}

\begin{proof}
    %% TV -> Best polynomial approximation error
Denote the best degree-$L$ polynomial approximation error of a function $f$ on an interval $I$ by
\[
E_L(f,I) = \inf_{p \in \calP_L} \sup_{x \in I} |f(x)-p(x)|.
\]  
    Let
    \begin{equation}
        f_j(x)\triangleq \frac{e^{-x}x^j}{j!}.
        \label{eq:fj}
    \end{equation}
    Let $ P_{L,j}^* $ be the best polynomial of degree $ L $ that uniformly approximates $ f_j $ over the interval $ [0,\Lambda] $ and the corresponding approximation error by $ E_L(f_j,[0,\Lambda])=\max_{x\in[0,\Lambda]}|f_j(x)-P_{L,j}^*(x)| $. 
    % and $ \calS_L=\sth{(V,V')\in[0,\Lambda]^2:\Expect[V^i]=\Expect[V'^i], i=1,\dots,L} $.
    Then $\Expect P_{L,j}^*(V) = \Expect P_{L,j}^*(V') $ and hence
    \begin{align}
      \TV(\Expect[\Poi(V)],\Expect[\Poi(V')])        =&~  \frac{1}{2}\sum_{j=0}^{\infty}|\Expect f_j(V)-\Expect f_j(V')|\nonumber\\
      \le &~\frac{1}{2}\sum_{j=0}^{\infty}|\Expect (f_j(V)-P_{L,j}^*(V))|+|\Expect (f_j(V')-P_{L,j}^*(V'))| \nonumber \\
      \le &~ \sum_{j=0}^{\infty} E_L(f_j,[0,\Lambda]).\label{eq:tv-ub-approx}
    \end{align}
    % in view of the relation of moment matching and best polynomial approximation in \cite[Appendix E]{WY14}. \nb{add standalone proof}

    %% Best polynomial approximation error -> Chebyshev approximation -> bound through derivatives -> Laguerre polynomial
    A useful upper bound on the degree-$L$ best polynomial approximation error of a function $f$ is via the Chebyshev interpolation polynomial, whose uniform approximation error can be bounded using the $L\Th$ derivative of $f$. Specifically, we have (cf.~\eg, \cite[Eq.~(4.7.28)]{Atkinson89})
    % \cite[Lecture 20]{stewart1993})
    \begin{equation}
        E_L(f,[0,\Lambda])
        \le \max_{x\in[0,\Lambda]}\abs{f_j(x)-Q_L(f;x)}
        \le \frac{1}{2^L(L+1)!}\pth{\frac{\Lambda}{2}}^{L+1}\max_{x\in [0,\Lambda]}\abs{f^{(L+1)}(x)},
        \label{eq:approx-ub-diff}
    \end{equation}
    where $ Q_L(f;x) $ denotes the degree-$ L $ interpolating polynomial for $f$ on Chebyshev nodes (roots of the Chebyshev polynomial). To apply \prettyref{eq:approx-ub-diff} to $f=f_j$ defined in \prettyref{eq:fj}, note that
    $ f_j^{(L+1)} $ can be conveniently expressed in terms of Laguerre polynomials: 
    Denote the degree-$ n $ generalized Laguerre polynomial by $ L_n^{(k)} $ and the simple Laguerre polynomial by $ L_n(x)=L_n^{(0)} $.
    Recall the Rodrigues representation for Laguerre polynomials:
    \[
        L_n^{(k)}(x)=\frac{x^{-k}e^x}{n!}\fracdk{}{x}{n}(e^{-x}x^{n+k})=(-1)^k\fracdk{}{k}{x}L_{n+k}(x), \quad k\in\naturals .
    \]

    %% Upper bound when j<= L+1
    If $ j\le L+1 $, 
    \begin{equation*}
        f_j^{(L+1)}(x)
        =\fracdk{}{x}{L+1-j}\pth{\fracdk{}{x}{j}\frac{e^{-x}x^j}{j!}}
        =\fracdk{}{x}{L+1-j}(L_j(x)e^{-x}).
    \end{equation*}
    Note that $ L_j $ is a degree-$ j $ polynomial, whose derivative of order higher than $ j $ is zero.
    Applying general Leibniz rule for derivatives yields that
    \begin{align}
      f_j^{(L+1)}(x)
      =~&\sum_{m=0}^{(L+1-j)\wedge j}\binom{L+1-j}{m} \fracdk{L_j(x)}{x}{m} e^{-x}(-1)^{L+1-j-m}\nonumber\\
      =~&(-1)^{L+1-j}e^{-x}\sum_{m=0}^{(L+1-j)\wedge j}\binom{L+1-j}{m}L_{j-m}^{(m)}(x).
          \label{eq:fj-diff}
    \end{align}
    Applying  \cite[22.14.13]{AS64}
\begin{equation}
|L_n^{(k)}(x)|\le \binom{n+k}{n}e^{x/2}         
        \label{eq:AS64}
\end{equation}
when $ x\ge 0 $ and $ k\in \naturals $, we obtain that
    \begin{equation*}
        \abs{f_j^{(L+1)}(x)}
        \le e^{-x}\sum_{m=0}^{(L+1-j)\wedge j}\binom{L+1-j}{m}\binom{j}{j-m}e^{x/2}
        = e^{-x/2}\binom{L+1}{j}.
    \end{equation*}
    Therefore $ \max_{x\in [0,\Lambda]}|f_j^{(L+1)}(x)|\le \binom{L+1}{j} $ when $ j\le L+1 $.\footnote{This is in fact an equality. In view of \prettyref{eq:fj-diff} and the fact that 
    \cite[22.3]{AS64}, we have $ |f_j^{(L+1)}(0)|=\sum_{m}\binom{L+1-j}{m}\binom{j}{j-m}=\binom{L+1}{j} $.}
    Then, applying \prettyref{eq:approx-ub-diff}, we have
    \begin{equation}
        \sum_{j=0}^{L+1} E_L(f_j,[0,\Lambda])
        \le \sum_{j=0}^{L+1}\frac{\binom{L+1}{j}(\Lambda/2)^{L+1}}{2^L(L+1)!}
        = \frac{2(\Lambda/2)^{L+1}}{(L+1)!}.
        % < \frac{2}{\sqrt{2\pi(L+1)}}\pth{\frac{e\Lambda}{2(L+1)}}^{L+1},
        \label{eq:approx-bd-jsmall}
    \end{equation}
    % where we used Stirling's approximation that $ n!>\sqrt{2\pi n}(n/e)^n $.

    %% Upper bound when j>L+1.
    If $ j\ge L+2 $, the derivatives of $ f_j $ are related to the Laguerre polynomial by
    \begin{equation*}
        f_j^{(L+1)}(x)=\frac{(L+1)!}{j!}x^{j-L-1}e^{-x}L_{L+1}^{(j-L-1)}(x).
    \end{equation*}
    Again applying \prettyref{eq:AS64} when $ x\ge 0 $ and $ k\in \naturals $, we obtain 
    \begin{equation*}
        \abs{f_j^{(L+1)}(x)}
        \le \frac{(L+1)!}{j!}x^{j-L-1}e^{-x}\binom{j}{L+1}e^{x/2}
        = \frac{1}{(j-L-1)!}e^{-x/2}x^{j-L-1},
    \end{equation*}
    where the maximum of right-hand side on $[0,\Lambda]$ occurs at $ x=(2(j-L-1))\wedge \Lambda $.
    Therefore
%    we obtain the following upper bound of $ \max_{x\in[0,\Lambda]}|f_j^{(L+1)}(x)| $
    \begin{equation*}
        \max_{x\in[0,\Lambda]}|f_j^{(L+1)}(x)|
        \le
        \begin{cases}
            \frac{1}{(j-L-1)!}\pth{\frac{2(j-L-1)}{e}}^{j-L-1},& L+1\le j\le L+1+\Lambda/2,\\
            \frac{1}{(j-L-1)!}e^{-\Lambda/2}\Lambda^{j-L-1},& j\ge L+1+\Lambda/2.
        \end{cases}
    \end{equation*}
    Then, applying \prettyref{eq:approx-ub-diff} and Stirling's approximation that $ (\frac{j-L-1}{e})^{j-L-1}<\frac{(j-L-1)!}{\sqrt{2\pi(j-L-1)}} $, we have
    \begin{align}
      \sum_{\substack{j\ge L+2\\ j<L+1+\Lambda/2}}E_L(f_j,[0,\Lambda])
   & \le \frac{(\Lambda/2)^{L+1}}{2^L(L+1)!}\sum_{\substack{j\ge L+2\\ j<L+1+\Lambda/2}}\frac{2^{j-L-1}}{\sqrt{2\pi(j-L-1)}}
      \le \frac{(\Lambda/2)^{L+1}2^{\Lambda/2}}{2^L(L+1)!},
      \label{eq:approx-bd-jmid}\\
      \sum_{j\ge L+1+\Lambda/2}E_L(f_j,[0,\Lambda])
   & \le \frac{(\Lambda/2)^{L+1}e^{-\Lambda/2}}{2^{L}(L+1)!}\sum_{j\ge L+1+\Lambda/2}\frac{\Lambda^{j-L-1}}{(j-L-1)!}
     \le \frac{(\Lambda/2)^{L+1}e^{\Lambda/2}}{2^{L}(L+1)!}.
     \label{eq:approx-bd-jbig}
    \end{align}
    % Applying Stirling's approximation again that $ n!>\sqrt{2\pi n}(n/e)^n $ in \prettyref{eq:approx-bd-jmid} -- \prettyref{eq:approx-bd-jbig} and
    Assembling the three ranges of summations in \prettyref{eq:approx-bd-jsmall}-\prettyref{eq:approx-bd-jbig} in the total variation bound \prettyref{eq:tv-ub-approx}, we obtain
    \begin{equation*}
        \TV(\Expect[\Poi(V)],\Expect[\Poi(V')])
        \le \frac{(\Lambda/2)^{L+1}}{(L+1)!}\pth{2+2^{\Lambda/2-L}+2^{\Lambda/(2\log 2)-L}}.
        % < \frac{1}{\sqrt{2\pi(L+1)}}\pth{\frac{e\Lambda}{2(L+1)}}^{L+1}\pth{2+2^{\Lambda/2-L}+2^{\Lambda/(2\log 2)-L}}.
    \end{equation*}
    
    Finally, applying Stirling's approximation $ (L+1)!>\sqrt{2\pi (L+1)}(\frac{L+1}{e})^{L+1} $, we conclude $ \TV(\Expect[\Poi(V)],\Expect[\Poi(V')])\le (\frac{e\Lambda}{2L})^{L} $.
%% reasoning
%    If L>e Lambda / 2, apply stirling's approximation that (L+1)!>\sqrt(2 pi (L+1))(L+1/e)^(L+1).
%RHS of (56) <=(e Lambda/2(L+1))^{L+1}  * (2+2^{L(1/e-1)}+2^{L(1/e log 2 - 1)})/   sqrt{2 pi * 2}
%< (e Lambda/2L)^{L+1} < (e Lambda/2L)^{L}
%where (2+2^{L(1/e-1)}+2^{L(1/e log 2 - 1)})/   sqrt{2 pi * 2}<= (2+2^{(1/e-1)}+2^{(1/e log 2 - 1)})/   sqrt{2 pi * 2} = 0.949971...
%If L<= e Lambda /2, then (e Lambda/2 L)^L>=1 which is a trivial UB. 
    If $ L>\frac{e}{2}\Lambda>\frac{\Lambda}{2\log 2}>\frac{\Lambda}{2} $, then $ 2^{\Lambda/2-L}+2^{\Lambda/(2\log 2)-L}=o(1) $.    
    % When $ j\ge L> \Lambda $, we also have a trivial upper bound when $ f_j $ is approximated by zero function that
    % \begin{equation*}
    %     E_L(f_j,[0,\Lambda])
    %     \le \max_{x\in[0,\Lambda]}|f_j(x)|
    %     = \frac{e^{-\Lambda}\Lambda^j}{j!}.
    % \end{equation*}
\end{proof}

%\begin{remark}
%    Recall that in the analysis we conclude that $ \max_{x\in [0,\Lambda]}|f_j^{(L+1)}(x)|= \binom{L+1}{j} $ when $ j\le L+1 $, then 
%    \begin{equation*}
%        \frac{(\Lambda/2)^{L+1}}{2^L(L+1)!}\sum_{j= 0}^{L+1}\max_{x\in [0,\Lambda]}\abs{f_j^{(L+1)}(x)}
%        = \frac{2(\Lambda/2)^{L+1}}{(L+1)!}.
%    \end{equation*}
%    This is the best possible upper bound if we use \prettyref{eq:approx-ub-diff} to upper bound the uniform approximation error $ E_L(f_j,[0,\Lambda]) $ when $ j\le L+1 $.
%\end{remark}

\section*{Acknowledgment}
% \label{sec:}
%This research has been supported in part by NSF grants IIS-14-47879, CCF-14-23088, CCF-15-27105 and the 
%Strategic Research Initiative of the College of Engineering at the University of Illinois. 
This work was completed in part when Y.W.~was visiting the Simons Institute for the Theory
of Computing, whose generous support is acknowledged.
The authors thank Luca Trevisan for helpful comments pertaining to \prettyref{thm:sample}.
The authors are grateful to Dan Roth and Mark Sammons for help with the datasets used in \prettyref{fig:NYT}.